\documentclass[11pt,a4paper]{amsart}
\usepackage{amsfonts}
\usepackage{amsmath}
\usepackage{amssymb, esint}
\usepackage{amsthm,color}
\usepackage{enumerate}
\usepackage{mathtools}
\usepackage{geometry}
\usepackage{graphicx}
\usepackage{mathrsfs}
\usepackage{hyperref}
\hypersetup{colorlinks=true, linkcolor=blue, citecolor=blue, urlcolor=blue}

\theoremstyle{plain}
\newtheorem{theorem}{Theorem}

\newtheorem{lemma}[theorem]{Lemma}

\theoremstyle{definition}
\newtheorem{remark}[theorem]{Remark}

\title[Complex sine-Gordon]{Symmetry breaking for the complex sine-Gordon equation}

\author[S.B. Chen]{Shibing Chen}
\address{Shibing Chen, School of Mathematical Sciences,
University of Science and Technology of China,
Hefei, Anhui 230026, China}
\email{chenshib@ustc.edu.cn}

\author[C.F. Gui]{Changfeng Gui}
\address{Changfeng Gui, Department of Mathematics, Faculty of Science and Technology, University of Macau, Macau, P.R. China}
\email{changfenggui@um.edu.mo}

\author[Y. Liu]{Yong Liu}
\address{Yong Liu, School of Mathematics and Statistics, Beijing Technology and Business University, Beijing, China}
\email{yliumath@btbu.edu.cn} 

\author[J.M. Yang]{Jianmin Yang}
\address{Jianmin Yang, Department of Mathematics, Faculty of Science and Technology, University of Macau, Macau, P.R. China}
\email{yangjianmin@mail.ustc.edu.cn} 

\author[W. Yang]{Wen Yang}
\address{\noindent Wen ~Yang,~Department of Mathematics, Faculty of Science and Technology, University of Macau, Macau, P.R. China}
\email{wenyang@um.edu.mo}



\date{\today}
\dedicatory{}
\begin{document}
\begin{abstract}
    We consider the existence of vortex solutions to the complex sine-Gordon II (CSG2) equation, which can be viewed as an analogy of the Ginzburg-Landau (GL) equation. Using the nontrivial kernels $\eta_{\pm}$ of the linearized CSG2 equation at the standard degree-2 vortex solution $\Psi_2$, we show that it bifurcating to a one-parameter family of symmetry-breaking solutions $\Psi_{2,\alpha}$. Explicit formulas of these solutions are also available, from which we propose a new bilinear system for this equation. Our method can be generalized to higher degree case. Nondegenracy and stability of degree-1 solution are also proved. Finally, we formally discuss the Lyapunov-Schmidt reduction procedure for the multivortex solutions of the CSG2 and GL equation.

\end{abstract}

\maketitle

\section{Introduction}

The complex sine-Gordon (CSG) equation originates from the Lund-Regge model \cite{LR76}. It mainly occurs in the following two forms, referred to as the CSG1 equation and the CSG2 equation, respectively:
\begin{equation}\label{csg1}
    \Delta \psi+\frac{(\nabla \psi)^2 \bar{\psi}}{1-|\psi|^2}+\psi\left(1-|\psi|^2\right)=0,
\end{equation}
\begin{equation}\label{csg2}
    \Delta \psi+\frac{(\nabla \psi)^2 \bar{\psi}}{2-|\psi|^2}+\frac{1}{2} \psi\left(1-|\psi|^2\right)\left(2-|\psi|^2\right)=0, 
\end{equation}
where, in both cases, $\psi$ is a complex-valued function defined on $\mathbb{R}^2$. Here and below, we set $\nabla=(\partial_{x},\partial_{y})$ and $z=re^{i \theta}=x+iy$. We agree that $\psi(z)=\psi(x,y)$. Note that 
$$
\partial_x=\cos\theta \partial_r-\frac{\sin\theta}{r}\partial_{\theta},\quad \partial_y=\sin\theta \partial_r+\frac{\cos\theta}{r}\partial_{\theta},
$$
and
$$
(\nabla \psi)^2=(\partial_x \psi)^2+(\partial_y \psi)^2=(\partial_r \psi)^2+\frac{1}{r^2}(\partial_{\theta}\psi)^2.
$$
If we assume that $\psi$ is real-valued and set $\psi=\sin (u/2)$ in Eq.~\eqref{csg1} and $\psi=\sqrt{2}\sin(u/4)$ in Eq.~\eqref{csg2}, then both equations reduce to the so-called elliptic sine-Gordon equation
\begin{equation}\label{sg}
    \Delta u+\sin u=0.
\end{equation}
Eq.~\eqref{sg} primarily describes a real scalar field, whereas the CSG equation involves a complex-valued field and therefore naturally carries phase, modulus, topological charge, and vortex structures. It has been widely studied in two-dimensional field theory, nonlinear sigma models, and the theory of solitons and vortices; see \cite{BG87,BG93,BGG03,BP98,BSA02,G77,G80,OB05}. 

The CSG2 equation can be viewed as an integrable analogy of the Ginzburg-Landau (GL) equation
\begin{equation}\label{gl}
    \Delta \psi+\psi (1-|\psi|^2)=0,\quad \psi:\mathbb{R}^2 \longmapsto \mathbb{C}.
\end{equation}
This equation has a family of ``standard vortex solutions''. They are of the form $\psi=f_n(r) e^{i n \theta}$ with $f_n$ being monotone increasing and
$$
f_n(r)\to 1\quad \text{as } r\to \infty.
$$   Here $n$ is called the degree of the solution.
Note that this solution is symmetric, which means that its modulus only depends on the radial coordinate $r$. 

A natural question is whether the Ginzburg-Landau equation has other vortex solutions that break the radial symmetry of the standard one. 
Bethuel-Brezis-H\'elein (BBH) \cite{B23,BBH94}  (Open problem 2.4) conjectured that all the degree $n$ vortex solutions have to be radially symmetric, up to translation. This conjecture has been proved for $|n|\leq 1$ under the additional assumption of finite potential energy, see \cite{BMR94,M96}. For $|n|\geq 2$, however, the conjecture remains unresolved.  It turns out that this conjecture is very difficult. This is the primary motivation for our study of the CSG2 equation here. Since CSG2 is, in a suitable sense, an integrable system, we would like to see whether there are any symmetry-breaking vortex solutions for the CSG2 equation, and we hope that results of CSG2 obtained in this paper can provide useful insights for the understanding of the GL equation. We also mention that the existence or nonexistence of symmetry-breaking solutions is closely related to the kernels of the linearized operator around the radially symmetric vortex solutions. 

At this stage, it is worth remarking that from the variational point of view, the GL equation is closer to CSG2 equation than to CSG1 equation. To explain this in more details, we recall that Eqs.~\eqref{csg1}, \eqref{csg2} and \eqref{gl} can be viewed as the critical point equations of the following energy functionals:
\begin{equation*}
    \begin{aligned}
\mathcal{E}_{\mathrm{CSG1}}\left( \psi \right) &=\int{\left[ \frac{|\nabla \psi |^2}{1-|\psi |^2}+\left( 1-|\psi |^2 \right) \right] dx dy},\\
\mathcal{E}_{\mathrm{CSG2}}\left( \psi \right) &=\int{\left[ \frac{|\nabla \psi |^2}{2-|\psi |^2}+\frac{1}{4}\left( 1-|\psi |^2 \right)^2 \right] dx dy},\\
\mathcal{E}_{\mathrm{GL}}\left( \psi \right) &=\int{\left[ \frac{1}{2}|\nabla \psi |^2+\frac{1}{4}\left( 1-|\psi |^2 \right)^2 \right] dx dy},
    \end{aligned}
\end{equation*}
respectively.  Symmetry breaking solutions of the CSG1 equation expressed in terms of modified Bessel functions are available; see \cite{BSA02}. The moduli of these solutions oscillate around $1$ infinitely many times and the solutions actually have infinite potential energy. We also note that when $|\psi |=1$, the denominator in the energy functional vanishes. On the other hand, for the radial degree-$n$ vortex solution of the GL equation, the second term in the energy functional is finite and no oscillation occurs for that family of solutions. In view of the explicit form of the CSG2 energy functional, we expect CSG2 to admit finite-potential-energy symmetry-breaking solutions whose moduli are strictly less than $1$. In particular, the denominator of the first term in the energy functional will be positive in this case. Moreover, in the degree one case, one also expects that such solution is energy minimizing, a famous property already known for the Ginzburg-Landau equation.

It is worth noting that Ovchinnikov-Sigal \cite{OS97} proposed an energy-based method for constructing non-radial solutions of the GL equation on $\mathbb{R}^2$, relying on a nonvanishing condition for a second-order coefficient in the energy expansion, known as the correlation coefficient. Unfortunately, for the vortex configurations considered by Ovchinnikov-Sigal, Esposito \cite{E13} proved that the corresponding correlation coefficient vanishes. Kurzke \cite{K19} subsequently showed that, for any equilibrium vortex configuration, the corresponding correlation coefficient is identically zero. Thus, for $|n|\geq 2$, the BBH conjecture remains an open problem. Our analysis in this paper shows that the Ovchinnikov-Sigal solution do exist  for the complex sine-Gordon II equation, while also suggests that symmetry-breaking solutions of the Ginzburg-Landau equation should not exist.

A systematic mathematical theory for the singular limit of two-dimensional GL vortices was developed by Bethuel-Brezis-Hélein \cite{BBH94}. They showed that energy concentrates near finitely many vortex points and that the leading interaction between these points is described by a renormalized energy. Lin \cite{Lin95,Lin95-2} subsequently established a relation between solutions of the GL equation and critical points of the renormalized energy and investigated nonminimizing configurations containing vortices and antivortices. These results show that, away from the vortex cores, a multivortex configuration can be described by an effective finite-dimensional point-vortex system. The stability and minimality of radial vortices was studied by Mironescu \cite{M95}. Del Pino-Felmer-Kowalczyk \cite{DFK04} proved the non-degeneracy and stability of the standard degree-1 vortex using Hardy type inequality and obtained a Fredholm theory for the corresponding linearized operator. Such spectral properties are fundamental for Lyapunov-Schmidt constructions of multivortex solutions, see \cite{DKM06}.

For the existence and stability theory of multivortex solutions to the GL equation with the magnetic field, we refer to \cite{GS00,GS06,PTW12,T80,TW13,WW21}. Similar to the deep connection between the Allen-Cahn equation 
\begin{equation*}
    \Delta u+u-u^3=0\quad \text{in }\mathbb{R}^N,
\end{equation*}
and minimal surface theory (see \cite{CW18, P12}), for the magnetic GL equation one can construct solutions whose zero sets concentrate near codimension-two minimal submanifolds and carry out the corresponding stability analysis; see \cite{BD23,BD24,BD24-2,LMWW25,LWY24}. 
Here we point out that the relation between the CSG2 equation and the GL equation is quite similar to the relation between Allen-Cahn equation and the elliptic sine-Gordon equation (see \cite{LW21}).

For the Gross-Pitaevskii (GP) equation
\begin{equation}\label{GP}
    i \partial_t \psi+\Delta \psi+\left(1-|\psi|^2\right) \psi=0,
\end{equation}
which is closely related to the  Ginzburg-Landau equation, the vortex dynamics has also been widely studied. For the construction of traveling waves and vortex ring solutions to Eq.~\eqref{GP}, we refer to \cite{AHLW21,BOS04,BS99,DDMR22,LW20}. Note that the positions of the traveling waves in \cite{AHLW21,LW20} are governed by suitable pairs of Adler-Moser polynomials, and the corresponding non-degeneracy is crucial for handling the projection equations within the Lyapunov-Schmidt reduction framework. Interestingly enough, it turns out that, from the expression of our symmetry-breaking solutions, the Adler-Moser polynomials \cite{Adler} also appear in the complex sine-Gordon II equation.

The main contribution of this paper is to provide explicit expressions and an algorithm for families of multivortex solutions of the CSG2 equation in the plane. 

To better explain our results, now let us introduce the radial vortex solutions of the CSG2 equation.
The standard degree-1 and degree-2 vortex solutions of the CSG2 equation take the following form
$$
\Psi_1=\sqrt{Q_1}e^{i \theta}\quad\text{and} \quad\Psi_2=\sqrt{Q_2}e^{2i \theta},
$$
where 
$$
Q_{1}=\frac{r^{2}}{r^{2}+4}\quad\text{and} \quad Q_{2}=\frac{r^{4}\left(  r^{2}+24\right)  ^{2}}{r^{8}+64r^{6}+1152r^{4}+9216r^{2}+36864}.
$$
We set
$$
\Phi_n(r)=|\Psi_n(z)|=\sqrt{Q_n(r)}.
$$
In fact, using the Schlesinger transformation for the Painlev\'e-V equation, also called the vorticity-raising transformation in \cite{OB05}, Olver-Barashenkov derived the explicit formula for the radial vortex solution of the CSG2 equation with arbitrary vorticity $n$.  

Note that 
\begin{align*}
e^{i\theta} &  =\cos\theta+i\sin\theta=\frac{x+yi}{r},\\
e^{2i\theta} &  =\left(  \frac{x+yi}{r}\right)  ^{2}=\frac{x^{2}-y^{2}%
+2xyi}{r^{2}}.
\end{align*}
Therefore,
\begin{align*}
    \Psi_{1}&=\sqrt{\frac{r^{2}}{r^{2}+4}}e^{i\theta}=\frac{r}{\sqrt{r^{2}+4}}
\frac{x+yi}{r}=\frac{x+yi}{\sqrt{r^{2}+4}},\\
\Psi_{2} &  =\sqrt{\frac{r^{4}\left(  r^{2}+24\right)  ^{2}}{r^{8}%
+64r^{6}+1152r^{4}+9216r^{2}+36864}}e^{2i\theta}\\
&  =\frac{r^{2}\left(  r^{2}+24\right)  }{\sqrt{r^{8}+64r^{6}+1152r^{4}%
+9216r^{2}+36864}}\frac{x^{2}-y^{2}+2xyi}{r^{2}}\\
&  =\frac{\left(  r^{2}+24\right)  \left(  x^{2}-y^{2}+2xyi\right)  }%
{\sqrt{r^{8}+64r^{6}+1152r^{4}+9216r^{2}+36864}}.
\end{align*}
Note that $\Phi_1(r)$ has the following asymptotic behavior 
\begin{equation}\label{asy}
    \begin{aligned}
        \Phi_1(r)&=\frac{1}{2}r+O(r^3)\quad \text{as } r\to 0,\\
        \Phi_1(r)&=1-\frac{2}{r^2}+O\left(\frac{1}{r^4}\right)\quad \text{as } r\to \infty,\\
        \Phi'_1(r)&=\frac{4}{r^3}+O\left(\frac{1}{r^5}\right)\quad \text{as } r\to \infty.
    \end{aligned}
\end{equation}
We expect the existence of a family of degree-2 vortex solutions near $\Psi_{2}$. One of our key observations is that the
precise form of these solutions should have the following form
\begin{equation}\label{degree2}
    \begin{aligned}
\Psi _{2,\alpha}&=\frac{\beta _{1,\alpha}\left( x^2-y^2-\alpha x+\alpha y \right) +\beta _{2,\alpha}\left( 2xy+\alpha x+\alpha y \right)}{\sqrt{P_{\alpha}}}\\
&+\frac{\gamma _{1,\alpha}\left( 2xy+\alpha x+\alpha y \right) +\gamma _{2,\alpha}\left( x^2-y^2-\alpha x+\alpha y \right)}{\sqrt{P_{\alpha}}}i.
\end{aligned}
\end{equation}
Here
$$
\beta_{1,0}  =\gamma_{1,0}=r^{2}+24,\quad 
    \beta_{2,0}   =\gamma_{2,0}=0,
$$
and
$$
P_{0} =r^{8}+64r^{6}+1152r^{4}+9216r^{2}+36864.
$$
We will later explain why it is natural to conjecture that the numerator of $\Psi _{2,\alpha}$ admits the above perturbative form, based on the existence of the nontrivial kernel of $\Psi_2$. 

We can now state the main result of this paper, namely the existence of a one-parameter degree-2 family of symmetry breaking vortex solutions of the CSG2 equation. 
\begin{theorem}\label{thm}
    The CSG2 equation admits a nontrivial one-parameter family of degree-2 vortex solutions $\Psi _{2,\alpha}$ of type \eqref{degree2}, where
   \begin{equation}\label{bb}
    \beta _{1,\alpha}=r^2+24+\alpha \left( x+y \right) +2\alpha ^2,\quad \beta _{2,\alpha}=\alpha \left( -x+y \right) ,
   \end{equation}
\begin{equation}\label{gg}
    \gamma _{1,\alpha}=r^2+24-\alpha \left( x+y \right) +2\alpha ^2,\quad \gamma _{2,\alpha}=\alpha \left( -x+y \right) ,
\end{equation}
\begin{equation}\label{PP}
    \begin{aligned}
P_{\alpha}&=\sum_{k=0}^6{\alpha ^k\mathcal{P}_k}\\
&=r^8+64r^6+1152r^4+9216r^2+36864\\
&+\alpha \left( -48x^5+144x^4y+96x^3y^2-768x^3+96x^2y^3 \right.\\
&\quad\quad \left. +2304x^2y+144xy^4+2304xy^2-48y^5-768y^3 \right)\\
&+\alpha ^2\left( 1152r^2+4608 \right)\\
&+\alpha ^3\left( -4x^5+12x^4y+8x^3y^2-64x^3+8x^2y^3 \right.\\
&\quad\quad \left. +192x^2y+12xy^4+192xy^2-4y^5-64y^3 \right)\\
&+\alpha ^4\left( 192r^2+768 \right) +\alpha ^6\left(8 r^2+32 \right) .
    \end{aligned}
\end{equation}
\end{theorem}
\begin{remark}
The vortex points of this family precisely have the same configuration as that of Ovchinnikov-Sigal \cite{OS97}. Let us set
$$
R_{\alpha}=x^2-y^2-\alpha x+\alpha y,\quad I_{\alpha}=2xy+\alpha x+\alpha y,
$$
and denote the real and imaginary parts of the numerator of $\Psi_{2,\alpha}$ by
\begin{equation}\label{linear system}
    N_{1,\alpha}=\beta_{1,\alpha}R_{\alpha}+\beta_{2,\alpha}I_{\alpha},\quad N_{2,\alpha}=\gamma_{2,\alpha}R_{\alpha}+\gamma_{1,\alpha}I_{\alpha},
\end{equation}
respectively. Here $\beta_{j,\alpha}$ and $\gamma_{j,\alpha}$ are given by \eqref{bb}-\eqref{gg}. For $P_{\alpha}$ given in \eqref{PP}, using the numerical computation software \textit{Mathematica}, we obtain
\begin{equation*}
P_\alpha=(N_{1,\alpha})^2+(N_{2,\alpha})^2+16\Bigl(|z^3+\alpha(\alpha^2+12)(1-i)|^2+36|z|^4+576|z|^2+2304\Bigr)>0,
\end{equation*}
which implies that $|\Psi_{2,\alpha}|<1$.
The determinant of the linear system \eqref{linear system} in $(R_{\alpha},I_{\alpha})$ is 
\begin{equation*}
\begin{aligned}
\beta_{1,\alpha}\gamma_{1,\alpha}-\beta_{2,\alpha}\gamma_{2,\alpha}
&=(r^2+24+2\alpha^2)^2-2\alpha^2r^2\\
&=r^4+(48+2\alpha^2)r^2+(24+2\alpha^2)^2>0.
\end{aligned}
\end{equation*}
Thus $\Psi_{2,\alpha}(z)=0$ if and only if $R_{\alpha}(z)=I_{\alpha}(z)=0$. Consequently, the zeros of $\Psi_{2,\alpha}$ are given by the origin $p_0$, together with the following three points:
\begin{equation}\label{pj}
    \begin{aligned}
        p_1&=\sqrt{2}\alpha e^{-3\pi i/4}=-\alpha(1+i),\\
        p_2&=\sqrt{2}\alpha e^{-\pi i/12}=\frac{\alpha}{2}(\sqrt{3}+1)-\frac{i \alpha}{2}(\sqrt{3}-1),\\
        p_3&=\sqrt{2}\alpha e^{7\pi i/12}=-\frac{\alpha}{2}(\sqrt{3}-1)+\frac{i \alpha}{2}(\sqrt{3}+1).
    \end{aligned}
\end{equation}
We refer to such a vortex configuration as an equilateral triangular configuration.
Of course, by rotating the coordinates one obtains one-parameter families of equilateral triangular configurations with arbitrary orientation. 
\end{remark}

\begin{remark}
    Although all subsequent arguments are carried out for a nonnegative real parameter $\alpha$, the Hirota bilinear-derivative formulation in Section \ref{section4} shows that $\alpha$ can in fact be taken to be an arbitrary complex parameter.

We use the complex notation
\[
        z=x+iy,\quad \bar z=x-iy,\quad s=z\bar z=r^2 .
\] Let us introduce a complex parameter $\mu$ and denote by
    $$
    W_{\mu}=(|z|^2+24)z^2-2\mu \bar{z},\quad T_{\mu}=z^3+\mu,
    $$
$$
S_{\mu}=|T_{\mu}|^2+36|z|^4+576|z|^2+2304,\quad Q_{\mu}=|W_{\mu}|^2+16S_{\mu}.
$$
Define 
$$
\widetilde{\Psi}_{2,\mu}(z)=\frac{W_{\mu}(z)}{\sqrt{Q_{\mu}(z)}}.
$$
We find that $\widetilde{\Psi}_{2,\mu}$ is also a solution of the CSG2 equation with the $(1,3)$-configuration. Moreover, the family in Theorem \ref{thm} is just obtained by choosing 
\begin{equation}\label{mu}
    \mu=\alpha(\alpha^2+12)(1-i),
\end{equation}
and then
$$
W_{\mu}=N_{1,\alpha}+ i N_{2,\alpha},\quad Q_{\mu}=P_{\alpha}.
$$
Hence, under the choice of the complex parameter $\mu$ given in \eqref{mu}, we obtain
$$
\widetilde{\Psi}_{2,\mu}=\Psi_{2,\alpha}.
$$
Set 
$$
\eta=(\widetilde{\Psi}_{2,\mu})^2=\frac{g}{f}.
$$
Using the Hirota bilinear derivative operator $D$ defined in Section \ref{section4}, we find that the real function $f$ and the complex function $g$ satisfy the following system of bilinear equations:
\begin{equation}\label{eq:key-id-1}
        D^2 f\cdot f=512(S_{\mu})^2=2(f-|g|)^2,
\end{equation}
\begin{equation*}
        D^2 g\cdot g=2gh,
\end{equation*}
\begin{equation}\label{eq:key-id-3}
        D^2 g\cdot f=-16S_{\mu}(g+h),
\end{equation}
where $h=192 z (z^3+4\mu)$. Substituting formulas \eqref{eq:key-id-1}-\eqref{eq:key-id-3} into the right-hand side of formula \eqref{LHS}, and noting that the last line vanishes, we find that
\begin{equation*}
    \begin{aligned}
        \mathrm{LHS}\text{ of } \eqref{LHS}&=\frac{g\left( |g|-3f \right)}{2\left( 2f-|g| \right)}2f^2+\frac{\bar{g}}{|g|^2}\frac{f^2}{2f-|g|}g^2(3f-|g|)\\
        &=\frac{g f^2(|g|-3f)}{2f-|g|}+\frac{g f^2(3f-|g|)}{2f-|g|}\\
        &=0,
    \end{aligned}
\end{equation*}
which implies that $\widetilde{\Psi}_{2,\mu}$ solves the CSG2 equation. The derivation and transformation of the bilinear form of the equation can be found in Section \ref{section4}.
\end{remark}

Our method of proof also works for the degree 3 case, yielding a two-parameter family of exact degree-3
solutions of the CSG2 equation.  The family is naturally related to the
Adler--Moser polynomials and gives rise to two nontrivial bounded kernels of the linearized operator at the radial degree-3 vortex.
The radial degree-3 vortex solution is (see \cite{OB05}):
\begin{equation*}
\Psi_3(z)=\frac{z^3F_3(s)}{\sqrt{D_3(s)}},
\end{equation*}
where
\begin{equation}\label{eq:F3}
  F_3(s)=s^3+144s^2+5760s+92160,
\end{equation}
and
\begin{equation}\label{eq:D3}
    \begin{aligned}
        D_3(s)={}&s^9+324s^8+41472s^7+2820096s^6
        +114130944s^5 \\
&+2919628800s^4+50960793600s^3
        +611529523200s^2 \\
&+4892236185600s+19568944742400 .
    \end{aligned}
\end{equation}

\begin{theorem}[A two-parameter degree-3 family]
\label{thm:degree3-complex-family}
For every \(a,b\in\mathbb{C}\), the function
\begin{equation*}
    \Psi_{3,a,b}(z)=\frac{W_{a,b}(z,\bar z)}
    {\sqrt{P_{a,b}(z,\bar z)}}
\end{equation*}
is a smooth solution of the CSG2 equation \eqref{csg2}.  Its numerator is
\begin{align}
W_{a,b}={}&z^3F_3(s)
-\frac14\bar a\,z^6
-\frac54a\left(s^3+72s^2+1152s-9216\right) \notag\\
&-b(s+24)\bar z^2
+\frac14\bar a b\,z
+\frac5{16}a\bar a\,z^3
-\frac5{16}a^2\bar z^3
+\frac5{64}a^2\bar a .
\label{eq:N-complex}
\end{align}
The denominator \(P_{a,b}\) is the positive polynomial given by 
\begin{equation*}
\begin{aligned}
P_{a,b}={}&19568944742400
+4892236185600|z|^2
+611529523200|z^2|^2\\
&+50960793600|z|^6
+398131200|a|^2\\
&+2919628800\left|z^4-\frac1{22}az\right|^2
+\frac{298598400}{11}|az|^2\\
&+114130944\left|z^5-\frac{25}{172}az^2-\frac1{86}b\right|^2\\
&+\frac{74649600}{43}\left|az^2-\frac19b\right|^2\\
&+2820096\left|z^6-\frac{35}{102}az^3-\frac5{816}a^2
-\frac8{153}bz\right|^2\\
&+\frac{230400}{17}\left|az^3-\frac7{16}a^2-\frac13bz\right|^2\\
&+41472\left|z^7-\frac58az^4-\frac16bz^2\right|^2\\
&+324\left|z^8-az^5-\frac49bz^3+\frac1{36}ab\right|^2\\
&+\left|z^9-\frac32az^6+\frac5{64}a^3-bz^4+\frac14abz\right|^2 .
\end{aligned}
\end{equation*}
\end{theorem}

\begin{remark}[Relation with the Adler--Moser polynomials]
    The Adler--Moser polynomials relevant to this degree-3 family are
\begin{equation*}
    \Theta_2(z)=z^3+\tau_3,\qquad
    \Theta_3(z)=z^6+5\tau_3z^3+\tau_5z-5\tau_3^2 .
\end{equation*}
The parameters used above are related to the Adler--Moser parameters by
\begin{equation}\label{eq:tau-ab}
    \tau_3=-\frac{a}{4},\qquad \tau_5=-b .
\end{equation}
The leading weighted part of the numerator is governed by
\[
    \Theta_3(z)\overline{\Theta_2(z)} .
\]
More explicitly,
\begin{align*}
\Theta_3(z)\overline{\Theta_2(z)}
=&z^3s^3+\bar\tau_3z^6+5\tau_3s^3+\tau_5s\bar z^2
+\bar\tau_3\tau_5z\\
&+5\tau_3\bar\tau_3z^3-5\tau_3^2\bar z^3
-5\tau_3^2\bar\tau_3 .
\end{align*}
After inserting \eqref{eq:tau-ab}, we find that this is precisely the highest-weight (Parameters in Adler-Moser polynomials have weights, making the polynomial homogeneous)
structure of \eqref{eq:N-complex}.  The remaining lower-order radial terms in
\eqref{eq:N-complex} are the CSG2 dressing corrections. We believe this is a common feature for all degrees. Note that the roots of the numerator may not be same as the roots of the Adler-Moser polynomials.
\end{remark}

\begin{remark}[Nontrivial kernel of $\Psi_3$]
     Denote by the linearized CSG2 operator at a solution $\Psi$ by 
    \begin{equation*}
        \begin{aligned}
            \mathcal{L}_{\Psi}\eta=&\Delta\eta+\frac{2(\nabla\Psi\cdot\nabla\eta)\bar\Psi+(\nabla\Psi)^2\bar\eta}{2-|\Psi|^2}
+\frac{(\nabla\Psi)^2\bar\Psi}{(2-|\Psi|^2)^2}\left(\Psi\bar\eta+\bar\Psi\eta\right)\\
&+\frac12(2-3|\Psi|^2+|\Psi|^4)\eta
+\frac12\Psi(-3+2|\Psi|^2)\left(\Psi\bar\eta+\bar\Psi\eta\right).
        \end{aligned}
    \end{equation*}
Since \(\Psi_{3,a,b}\) is an exact family, differentiating
it with respect to the real and imaginary parts of \(a\) and
\(b\) at \((a,b)=(0,0)\) gives following bounded kernels of
\(\mathcal{L}_{\Psi_3}\):
\[
    \partial_a\Psi_{3,a,b}\big|_{a=b=0},\quad
    \partial_{\bar a}\Psi_{3,a,b}\big|_{a=b=0},\quad
    \partial_b\Psi_{3,a,b}\big|_{a=b=0},\quad
    \partial_{\bar b}\Psi_{3,a,b}\big|_{a=b=0}.
\]
We compute 
$$
\partial_a P_{a,b}\big|_{a=b=0} = -\bar z^{3}A(s), \quad \partial_{\bar a}P_{a,b}\big|_{a=b=0} = -z^3A(s),
$$
$$
\partial_b P\big|_{a=b=0} = -\bar z^{5}B(s),\quad \partial_{\bar b}P\big|_{a=b=0} = -z^5B(s),
$$
where 
$$
\begin{aligned}
    A(s)&=\frac{3}{2}s^6+324s^5+25920s^4+967680s^3
      +16588800s^2+132710400s,\\
B(s)&=s^4+144s^3+6912s^2+147456s+1327104.
\end{aligned}
$$
For $F_3,D_3$ given by \eqref{eq:F3}-\eqref{eq:D3}, we have  
\begin{equation*}
    \begin{aligned}
        \partial_a\Psi_{3,a,b}\big|_{a=b=0}
&=
\frac{
-\frac54\left(s^3+72s^2+1152s-9216\right)D_3(s)
+\frac12 s^3F_3(s)A(s)
}{
D_3(s)^{3/2}
},\\[0.4em]
\partial_{\bar a}\Psi_{3,a,b}\big|_{a=b=0}
&=
z^6
\frac{
-\frac14D_3(s)+\frac12F_3(s)A(s)
}{
D_3(s)^{3/2}
},\\[0.4em]
\partial_b\Psi_{3,a,b}\big|_{a=b=0}
&=
\bar z^2
\frac{
-(s+24)D_3(s)+\frac12s^3F_3(s)B(s)
}{
D_3(s)^{3/2}
},\\[0.4em]
\partial_{\bar b}\Psi_{3,a,b}\big|_{a=b=0}
&=
\frac12
z^8
\frac{
F_3(s)B(s)
}{
D_3(s)^{3/2}
},
    \end{aligned}
\end{equation*}
After factoring out the background phase $e^{3i\theta}$, we find 
the \(\bar{a},a\)-direction correspond to
the Fourier mode \(m=\pm3\), while the \(\bar{b},b\)-direction correspond to the Fourier mode of \(m=\pm 5\). For general degree $n$ solutions, those Fourier modes should be $3,5, ..., 2n-1$, yielding $n-1$ free complex parameters in the solution space. Note that in reality, nontrivial kernels are obtained before we found the exact form of the solutions.
\end{remark}

\begin{remark}
    The solution in Theorem \ref{thm:degree3-complex-family} can be verified using numerical software. Here we present the key transformed form of the equation. Write
\[
    \Psi=\frac{W}{\sqrt P},\quad
    \bar\Psi=\frac{M}{\sqrt P},
\]
where \(M\) is obtained from \(W\) by conjugating \(z,a,b\) into
\(\bar z,\bar a,\bar b\).  Treating \(z\) and \(\bar z\) as independent
variables, define
\[
    H=2P-WM,\quad K=P-WM,
\]
and
\[
    A=2PW_z-WP_z,\qquad B=2PW_{\bar z}-WP_{\bar z}.
\]
Multiplying the CSG2 equation by \(P^2\sqrt P\,H\), one obtains the
polynomial residual
\begin{equation}\label{eq:poly-residual-full}
    (2PA_{\bar z}-3AP_{\bar z})H+ABM+\frac12WKH^2 .
\end{equation}
For the polynomials \(W=W_{a,b}\) and \(P=P_{a,b}\) given in Theorem \ref{thm:degree3-complex-family}, this residual \eqref{eq:poly-residual-full} is identically zero.
\end{remark}

 We conjecture that the one-parameter family of solutions $\Psi_{2,\alpha}$ with the $(1,3)$-configuration constructed in this paper has Morse index 2. By way of comparison, Liu-Wei-Yang \cite{LWY24} used an inverse Lyapunov-Schmidt reduction argument to prove that the degree-6 lump solution of the KP-I equation (a class of classical integrable systems) obtained via the Bäcklund transformation has Morse index 4. It is worth noting that the lump locations considered there also form an equilateral-triangle configuration (but without a central point, consisting only of the three vertices), and that the corresponding vortex positions are also degenerate.

We briefly describe the strategy for finding the multivortex solution $\Psi_{2,\alpha}$ with the $(1,3)$-configuration in Theorem \ref{thm}, as well as the structure of the paper. In Section \ref{section2}, we first study the linearized operator of the CSG2 equation and find that it admits nontrivial kernels in the Fourier modes $m=\pm 3$. Starting from this observation, in Section \ref{section3}, we essentially use the structure of the roots of Adler-Moser polynomials to develop a method of undetermined coefficients which enables us to find an exact expression for $\Psi_{2,\alpha}$. This part is quite delicate, since the undetermined coefficients satisfy a highly nonlinear system of equations, and the computation is not easy. Our solutions cannot be obtained from the usual Hirota bilinear form of Getmanov.  Indeed, in Section \ref{section4}, we transform the CSG2 equation into Hirota bilinear-derivative form and check that the solution $\Psi_{2,\alpha}$ does not satisfy the bilinear-derivative system in \cite{G80}. We then propose a new bilinear system, which our solutions do satisfy. This will enable us to find higher degree solutions. In Section \ref{section5}, we study the non-degeneracy of $\Psi_1$ and the almost non-degeneracy (except for the Fourier modes $m=\pm 2$) of $\Psi_2$, as well as the stability of $\Psi_1$. Finally, in Section \ref{section6}, based on the Lyapunov-Schmidt reduction framework, we formally discuss the key ingredients in the construction of multivortex solutions for sufficiently large $\alpha$.

As a final remark, we would like to point out that now we can explicitly write down degree 2 and degree 3 solutions, and in view of this, it is expected (at least formally) that the space of all degree $n$ solutions has complex dimension $n$. That is, the solutions have $n$ free complex parameters. They should be associated to some polynomials of degree $2n^2$, and the vortex points should be determined by two consecutive Adler-Moser polynomials. A complete construction and classification of all these solutions would need a full understanding of the Inverse Scattering Transform of the complex sine-Gordon II equation, which is expected to be very similar to what we have done for the traveling lump solutions of the KP-I equation \cite{LW25}. However, this seems to be nontrivial at this moment. We hope to be able to return to this issue in the future.

\section{The nontrivial kernel of degree-2 vortex solution}\label{section2}

Let us consider the (2+1)-dimensional generalization of \eqref{csg2}:
\begin{equation}\label{CSG2-t}
\psi_{tt}-\Delta \psi-\frac{(\nabla \psi)^2 \bar{\psi}}{2-|\psi|^2}-\frac{1}{2} \psi\left(1-|\psi|^2\right)\left(2-|\psi|^2\right)=0.
\end{equation}
Assume $\Psi_n(z)=\Phi _n\left( r \right) e^{i n\theta}$ is the radial degree-$n$ vortex solution to \eqref{csg2}. Now we consider a solution of \eqref{CSG2-t} of the form  
\begin{equation}\label{purt}
\psi \left( r,\theta ,t \right) = \left[ \Phi _n\left( r \right) +\varepsilon \phi \left( r,\theta \right) \cos (\omega t) \right] e^{i n\theta}.
\end{equation}
Linearizing Eq.~\eqref{CSG2-t} in small $\varepsilon$, we obtain 
\begin{equation}\label{Ln}
    \begin{aligned}
\mathcal{L}_n\phi:=&-\nabla _{r}^{2}\phi -\frac{1}{r^2}\partial _{\theta}^{2}\phi-\frac{2\Phi _n\Phi'_n}{2-\Phi _{n}^{2}}\partial _r\phi -\frac{4i n}{r^2}\frac{1}{2-\Phi _{n}^{2}}\partial _{\theta}\phi\\
&+\left[ \frac{\left( 4n^2/r^2 \right) -\left( \Phi'_n \right) ^2\Phi _{n}^{2}}{\left( 2-\Phi _{n}^{2} \right) ^2}-\frac{3}{2}\Phi _{n}^{4}+3\Phi _{n}^{2}-1  \right] \phi \\
&+\left[ \frac{\left( 2n^2/r^2 \right) \Phi _{n}^{2}-2\left( \Phi'_n \right) ^2}{\left( 2-\Phi _{n}^{2} \right) ^2}-\Phi _{n}^{4}+\frac{3}{2}\Phi _{n}^{2} \right] \bar{\phi}=\omega ^2\phi,
    \end{aligned}
\end{equation}
where $\nabla _{r}^{2}=\partial^2 _{r}+r^{-1}\partial _r$. Expanding $\phi$ in the Fourier series in $\theta$:
\begin{equation}\label{fourier}
\phi(r, \theta)=\sum_{m=-\infty}^{\infty} \phi_m(r) e^{i  m \theta}=\sum_{m=-\infty}^{\infty}\left[a_m(r)+i  b_m(r)\right] e^{i  m \theta},
\end{equation}
and transforming to 
\begin{equation}\label{uv1}
    u_m(r)=a_m+a_{-m},\quad v_m(r)=a_m-a_{-m}.
\end{equation}
For each value of the azimuthal number $m$, we obtain a corresponding one-dimensional eigenvalue problem, resulting in a sequence of such problems:
\begin{equation}\label{eigen}
\mathcal{L}_{n, m}\binom{u_m}{v_m}=\omega^2\binom{u_m}{v_m},
\end{equation}
where the operator $\mathcal{L}_{n, m}$ is defined by 
\begin{equation}\label{Lmn}
    \mathcal{L}_{n,m}= \left( \begin{matrix}
	-\nabla _{r}^{2}+\mathfrak{B}_n\left( r \right) \frac{d }{d r}+\frac{m^2}{r^2}+\mathfrak{C}_n\left( r \right) +\mathfrak{D}_n\left( r \right)&		m\mathfrak{A}_n\left( r \right)\\
	m\mathfrak{A}_n\left( r \right)&		-\nabla _{r}^{2}+\mathfrak{B}_n\left( r \right) \frac{d }{d r}+\frac{m^2}{r^2}+\mathfrak{C}_n\left( r \right) -\mathfrak{D}_n\left( r \right)\\
\end{matrix} \right) 
\end{equation}
and
\begin{equation*}
    \begin{aligned}
\mathfrak{A}_n\left( r \right) &=\frac{4n}{r^2}\frac{1}{2-\Phi _{n}^{2}},\\
\mathfrak{B}_n\left( r \right) &=-\frac{2\Phi _n\Phi'_n}{2-\Phi _{n}^{2}},\\
\mathfrak{C}_n\left( r \right) &=\frac{n^2}{r^2}\frac{4}{\left( 2-\Phi _{n}^{2} \right) ^2}-\frac{\left( \Phi'_n \right) ^2\Phi _{n}^{2}}{\left( 2-\Phi _{n}^{2} \right) ^2}-\frac{3}{2}\Phi _{n}^{4}+3\Phi _{n}^{2}-1,\\
\mathfrak{D}_n\left( r \right) &=\frac{n^2}{r^2}\frac{2\Phi _{n}^{2}}{\left( 2-\Phi _{n}^{2} \right) ^2}-\frac{2\left( \Phi'_n \right) ^2}{\left( 2-\Phi _{n}^{2} \right) ^2}-\Phi _{n}^{4}+\frac{3}{2}\Phi _{n}^{2}.
    \end{aligned}
\end{equation*}
The imaginary parts of the Fourier coefficients $\phi_m$ satisfy the same eigenvalue problem \eqref{eigen} with $\mathcal{L}_{n,m}$ as in \eqref{Lmn}, where the eigenfunctions $(u_m,v_m)$ are defined by 
\begin{equation}\label{uv2}
    u_m(r)=b_m-b_{-m},\quad v_m(r)=b_m+b_{-m}.
\end{equation}
Keep in mind that the eigenfunctions $(u_m,v_m)$ depend on the vorticity $n$.

\begin{remark}\label{rmk1}
By the definition of \eqref{Lmn}, if $(u_m,v_m)$ 
is an eigenvector of $\mathcal{L}_{n,m}$ with eigenvalue $\omega^2$, then $(u_m,-v_m)$ is an eigenvector of $\mathcal{L}_{n,-m}$ with the same eigenvalue. Hence, it is sufficient to consider only the case where $m\geq  0$.
\end{remark}

Assume $\mathcal{L}_2\phi=0$, where the operator $\mathcal{L}_2$ is given by \eqref{Ln}. Recall the Fourier expansion in \eqref{fourier}.
For the Fourier mode $m=0$, corresponding to the $U(1)$-invariance of CSG2:
\begin{equation*}
e^{-2i \theta}\left.\frac{\partial}{\partial \alpha} \left[\Phi_2(r) e^{2i (\theta+\alpha)}\right]\right|_{\alpha=0}=2i \Phi_2(r)=i b_{0},
\end{equation*}
the kernel of $\mathcal{L}_{2,0}$, i.e. $\mathrm{Ker}\mathcal{L}_{2,0}$, contains the element $(u_0,v_0)^T=(0,4\Phi_2)^T$ by the relation \eqref{uv2}.
Similarly, consider the kernels generated by translational invariance (correspond to the Fourier mode $m=\pm1$):
\begin{equation*}
    \begin{aligned}
       e^{-2i \theta}\partial_{x}\left[\Phi_2(r)e^{2i \theta}\right]&=\frac{1}{2}\left(\Phi'_2(r)+\frac{2}{r}\Phi_2(r)\right) e^{-i \theta}+\frac{1}{2}\left(\Phi'_2(r)-\frac{2}{r}\Phi_2(r)\right)e^{i \theta}\\
       &=a_{-1} e^{-i \theta}+a_1  e^{i \theta},\\[0.5em]
       e^{-2i \theta}\partial_{y}\left[\Phi_2(r)e^{2i \theta}\right]&=\frac{i }{2}\left(\Phi'_2(r)+\frac{2}{r}\Phi_2(r)\right)e^{-i \theta}-\frac{i }{2}\left(\Phi'_2(r)-\frac{2}{r}\Phi_2(r)\right)e^{i \theta}\\
       &=i b_{-1}e^{-i \theta}+i b_{1}e^{i \theta}.
    \end{aligned}
\end{equation*}
Then by the relations \eqref{uv1} and \eqref{uv2}, we find that these kernels correspond, respectively, to the elements $(u_1,v_1)^T=(\Phi'_2,-(2/r)\Phi_2)^T\in\mathrm{Ker}\mathcal{L}_{2,1}$ and $(u_{-1},v_{-1})^T=(\Phi'_2,(2/r)\Phi_2)^T\in\mathrm{Ker}\mathcal{L}_{2,-1}$.

The following lemma highlights a fundamental distinction between degree-2 vortex solutions of the CSG2 equation and those of the Ginzburg-Landau equation, and suggests the existence of noncoaxial multivortex solutions for the CSG2 equation. 
\begin{lemma}\label{m3}
    At the Fourier mode $m=\pm 3$, the degree-2 vortex solution $\Phi_2$ admits a nontrivial kernel $(u_{\pm 3},v_{\pm 3})\in \mathrm{Ker}\mathcal{L}_{2,\pm 3}$, where 
    $$
u_{\pm 3}=\frac{-24r\left(r^6+32 r^4+384 r^2+1536 \right)}{\left( P_0 \right) ^{3/2}},\quad
v_{\pm 3}=\frac{\pm r}{\sqrt{P_0}}.
$$
\end{lemma}
\begin{proof}
    This can be verified directly using the numerical computation software \textit{Mathematica}.  
\end{proof}
\begin{remark}
    Here, we focus on the main idea behind the discovery of the family of nontrivial kernels in Lemma \ref{m3}. First, finite-difference computations indicate that a kernel is very likely to exist when $m=3$. Motivated by the explicit forms of the rotational and translational kernels, we conjecture that the denominator of $(u_3,v_3)$ should also contain the factor $\sqrt{P_0}$, while the remaining part should be rational.
By expanding the coefficients in \eqref{Lmn} for $n=2$ and $m=3$, after multiplying both sides by $r^2$ to avoid the singularity at the origin, together with $(u_3,v_3)$ into power series at the origin
$$
u_3=\sum_{k=0}^{\infty} \mathbf{u}_k r^k,\quad v_3=\sum_{k=0}^{\infty} \mathbf{v}_k r^k,  
$$
we can solve recursively for the coefficients $\mathbf{u}_k$ and $\mathbf{v}_k$. We find these coefficients depend on two parameters, $\mathbf{u}_1$ and $\mathbf{u}_5$, which correspond precisely to the two bounded kernel elements of $\mathcal{L}_{2,3}$ near the origin (see Lemma \ref{asy1}). Taking $\mathbf{u}_1=-1$, the left-hand side of \eqref{Lmn} depends only on the parameter $\mathbf{u}_5$. Solving then gives $\mathbf{u}_5=1/384$, for which the equation is satisfied. In this way, we obtain the Taylor coefficients of the nontrivial kernel $(u_3,v_3)$ at the origin. It is enough, however, to retain only finitely many terms in the expansion, for instance up to order 20.
We then use Pad\'e approximation to identify the rational part of $(u_3,v_3)$. It should be noted that this step also requires the asymptotic behavior of the kernel of $\mathcal{L}_{2,3}$ at infinity (see Lemma \ref{asy2}) in order to prescribe the degrees of the numerator and denominator of the rational function. Some numerical experimentation is needed here. We believe that this strategy may be useful for finding nontrivial kernels of the linearized CSG2 operator at other radial vortex solutions $\Psi_n$, where $n\geq 3$.
\end{remark}

By \eqref{uv1}, \eqref{uv2} and Lemma \ref{m3}, we set 
$$
a_3=\frac{u_3+v_3}{2},\quad a_{-3}=\frac{u_3-v_3}{2},\quad b_3=\frac{u_3+v_3}{2},\quad b_{-3}=\frac{v_3-u_3}{2},
$$
and define
\begin{equation}\label{eta+}
    \begin{aligned}
\eta_+ &=\left[ \left( a_3+i b_3 \right) e^{3i \theta}+\left( a_{-3}+i b_{-3} \right) e^{-3i \theta} \right] e^{2i \theta}\\
&=\frac{1+i }{2}\left( u_3+v_3 \right) e^{5i \theta}+\frac{1-i }{2}\left( u_3-v_3 \right) e^{-i \theta}\\
&=\frac{1+i }{2}\cdot\frac{r^5\left( 16+r^2 \right) \left( 24+r^2 \right)}{\left( P_0 \right) ^{3/2}}e^{5i \theta}\\
&-\frac{1-i }{2}\cdot\frac{r\left(r^8+88r^6+1920r^4+18432r^2+73728\right)}{\left( P_0 \right) ^{3/2}}e^{-i \theta}.
    \end{aligned}
\end{equation}
Denote by
$$
\eta_1= \left( \frac{u_3+v_3}{2} e^{3i \theta}+ \frac{u_3-v_3}{2}e^{-3i \theta} \right) e^{2i \theta},\quad \eta_2= i\left( \frac{u_3+v_3}{2} e^{3i \theta}+ \frac{v_3-u_3}{2}e^{-3i \theta} \right) e^{2i \theta}.
$$
Then 
$$
\eta_+=\eta_1+\eta_2\in \mathrm{Ker}\mathcal{L}_2.
$$
Note that if we use $(u_{-3},v_{-3})$ to represent the corresponding $a_{\pm 3}$ and $b_{\pm 3}$ by \eqref{uv1} and \eqref{uv2}, then 
$$
a_3=\frac{u_{-3}-v_{-3}}{2},\quad a_{-3}=\frac{u_{-3}+v_{-3}}{2},\quad b_3=\frac{v_{-3}-u_{-3}}{2},\quad b_{-3}=\frac{u_{-3}+v_{-3}}{2}.
$$
Recall that $u_3=u_{-3}$ and $v_{3}=-v_{-3}$.
Hence, we can obtain another element in the kernel of $\mathcal{L}_2$ that is linearly independent of $\eta_+$, namely,
\begin{equation}\label{eta-}
    \begin{aligned}
        \eta_-&=\frac{1-i }{2}\cdot\frac{r^5\left( 16+r^2 \right) \left( 24+r^2 \right)}{\left( P_0 \right) ^{3/2}}e^{5i \theta}\\
&-\frac{1+i }{2}\cdot\frac{r\left(r^8+88r^6+1920r^4+18432r^2+73728\right)}{\left( P_0 \right) ^{3/2}}e^{-i \theta}\\
&=\eta_1-\eta_2\in \mathrm{Ker}\mathcal{L}_2.
    \end{aligned}
\end{equation}
We now examine the properties of $\Phi_2$ under a small perturbation associated with $\eta_+$. Let us take $n=2,\varepsilon=0.1,\omega=0,\phi=\eta_+$ in \eqref{purt} and denote the corresponding deformation by $\psi(z;\eta_+)$. Figures \ref{fig1} and \ref{fig2} illustrate the profile of $|\psi(z;\eta_+)|$ and the corresponding distribution of its zeros (here we use polar coordinates for the plots, and hence the portion of the graph along the negative $x$-axis is not displayed). This strongly suggests that $\Phi_2$ belongs to a certain one-parameter family, and indicates the existence of noncoaxial multivortex solutions of the CSG2 equation. 

\begin{figure} [htbp]
    \centering
    \includegraphics[width=0.6\columnwidth, height=0.45\linewidth]{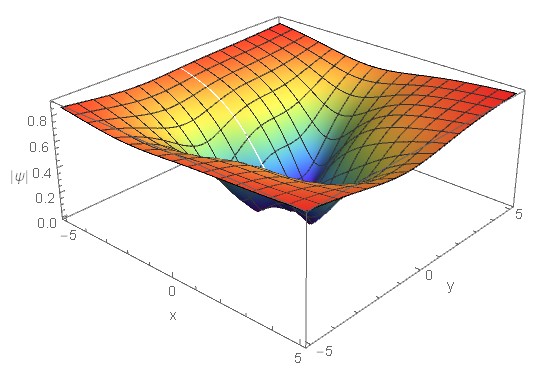}%
    \caption{The graph of $|\psi(z;\eta_+)|$.}
    \label{fig1}
\end{figure}

\begin{figure} [htbp]
    \centering
    \includegraphics[width=0.45\columnwidth, height=0.45\linewidth]{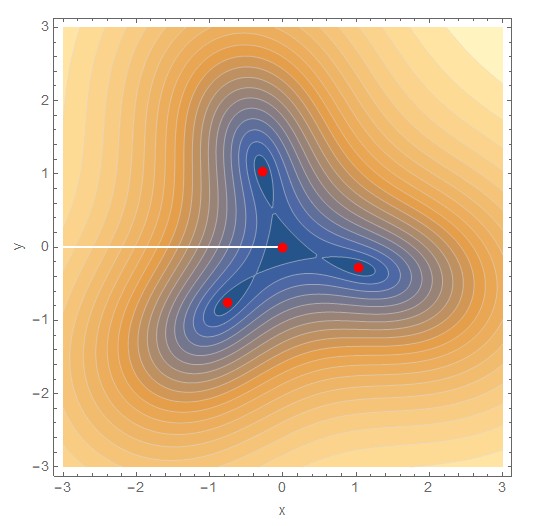}%
    \caption{The contour plot of $|\psi(z;\eta_+)|$. Note that the zeros of $|\psi(z;\eta_+)|$ form an equilateral triangular configuration.}
    \label{fig2}
\end{figure}

\bigskip
\section{A Family of Degree-2 Solutions to the complex sine-Gordon II Equation}\label{section3}

In this section, we will outline the strategy for finding the one-parameter family of degree-2 solutions described in Theorem \ref{thm}. 

Now we expect the zeros of the one-parameter family of solutions $\Psi _{2,\alpha}$ (see \eqref{degree2}) are determined by the following system
\begin{equation}\label{system1}
    \left\{ \begin{array}{l}
	x^2-y^2-\alpha x+\alpha y=0,\\
	2xy+\alpha x+\alpha y=0.\\
\end{array} \right. 
\end{equation}
This can be viewed as applying first-order polynomial perturbations to the real and imaginary parts of the factor $z^2$ in the numerator of $\Phi_2$, thereby desingularizing the double zero at $z=0$. It is straightforward to verify that the system \eqref{system1} admits the following four zeros 
$$
z=x+iy=0,\sqrt{2}\alpha e^{-3\pi i/4},\sqrt{2}\alpha e^{-\pi i/12},\sqrt{2}\alpha e^{7\pi i/12},
$$
which correspond precisely to zeros of the perturbation induced by $\eta_+$ for $\Phi_2$ (see Figure \ref{fig2}). This zero configuration will serve as the starting point for determining, via the method of undetermined coefficients, the polynomials corresponding to each order of $\alpha$ in the remaining factors of $\Psi _{2,\alpha}$. However, we must point out that the above choice of zero configuration is not unique; it may differ by an arbitrary complex scaling. For instance, consider the following system
\begin{equation}\label{system2}
    \left\{ \begin{array}{l}
	x^2-y^2+\alpha x=0,\\
	2xy-\alpha y=0.\\
\end{array} \right. 
\end{equation}
Then we find the system \eqref{system2} admits the following four zeros
$$
z=x+iy=0,\alpha e^{-\pi i},\alpha e^{-\pi i/3},\alpha e^{\pi i/3},
$$
which can be viewed as those of the system \eqref{system1} rotated clockwise by $\pi/4$ and scaled by a factor of $1/\sqrt{2}$.

We now use the zero-mode information of $\Psi_2$ to determine the polynomials corresponding to the first-order term in $\alpha$ in the remaining factors of $\Psi _{2,\alpha}$. 
By Lemma \ref{nond2}, there must exist five real constants $c_j,\ j=1,2,3,4,5$, such that
$$
\left.\partial_{\alpha}\Psi _{2,\alpha}\right|_{\alpha=0}=c_1 \partial_x \Psi_2+c_2 \partial_y \Psi_2+c_3 i\Psi_2+c_4 \eta_+ +c_5\eta_-.
$$
In fact, for different zero configurations of $\Psi _{2,\alpha}$ (equilateral triangular configurations centered at the origin), we will show that the only difference lies in the choice of $c_4$ and $c_5$, while $c_1,c_2,c_3$ all vanish. 

By \eqref{degree2}, we compute 
\begin{equation*}
\begin{aligned}
\operatorname{Re}\partial_{\alpha}\Psi_{2,\alpha}
&=
\frac{1}{P_{\alpha}}
\Bigg\{
\sqrt{P_{\alpha}}
\Big[
\partial_{\alpha}\beta_{1,\alpha}
\left(x^2-y^2-\alpha x+\alpha y\right)
+(y-x)\beta_{1,\alpha}
\Big] \\
&\qquad
-\frac{\partial_{\alpha}P_{\alpha}}{2\sqrt{P_{\alpha}}}
\,\beta_{1,\alpha}
\left(x^2-y^2-\alpha x+\alpha y\right)
\Bigg\} \\
&+
\frac{1}{P_{\alpha}}
\Bigg\{
\sqrt{P_{\alpha}}
\Big[
\partial_{\alpha}\beta_{2,\alpha}
\left(2xy+\alpha x+\alpha y\right)
+(x+y)\beta_{2,\alpha}
\Big] \\
&\qquad
-\frac{\partial_{\alpha}P_{\alpha}}{2\sqrt{P_{\alpha}}}
\,\beta_{2,\alpha}
\left(2xy+\alpha x+\alpha y\right)
\Bigg\}.
\end{aligned}
\end{equation*}
Then
\begin{equation*}
    \begin{aligned}
\text{Re}\partial _{\alpha}\Psi _{2,\alpha}|_{\alpha =0}&=\frac{2P_0\left[ \partial _{\alpha}\beta _{1,\alpha}|_{\alpha =0}\left( x^2-y^2 \right) +\left( -x+y \right) \left( r^2+24 \right) \right]}{2\left( P_0 \right) ^{3/2}}\\
&-\frac{\left( r^2+24 \right) \left( x^2-y^2 \right)\partial _{\alpha}P_{\alpha}|_{\alpha =0}}{2\left( P_0 \right) ^{3/2}}+\frac{2P_0\left[ \partial _{\alpha}\beta _{2,\alpha}|_{\alpha =0}\left( 2xy \right) \right]}{2\left( P_0 \right) ^{3/2}}.
    \end{aligned}
\end{equation*}
Similarly, we have 
\begin{equation*}
    \begin{aligned}
\text{Im}\partial _{\alpha}\Psi _{2,\alpha}|_{\alpha =0}&=\frac{2P_0\left[ \partial _{\alpha}\gamma _{1,\alpha}|_{\alpha =0}\left( 2xy \right) +\left( x+y \right) \left( r^2+24 \right) \right]}{2\left( P_0 \right) ^{3/2}}\\
&-\frac{\left( r^2+24 \right) \left( 2xy \right) \partial _{\alpha}P_{\alpha}|_{\alpha =0}}{2\left( P_0 \right) ^{3/2}}+\frac{2P_0\left[ \partial _{\alpha}\gamma _{2,\alpha}|_{\alpha =0}\left( x^2-y^2 \right) \right]}{2\left( P_0 \right) ^{3/2}}.
    \end{aligned}
\end{equation*}
Now let us consider the kernels $\partial_{x}\Psi_2$ and $\partial_{y}\Psi_2$. We
have
\begin{equation*}
    \begin{aligned}
        \partial_{x}\Psi_2 & =\frac{\partial_{x}\left[  \left(  r^{2}+24\right)  \left(
x^{2}-y^{2}+2xyi\right)  \right]  \sqrt{P_{0}}-\left[  \left(  r^{2}
+24\right)  \left(  x^{2}-y^{2}+2xyi\right)  \right]  \frac{\partial_{x}P_{0}
}{2\sqrt{P_{0}}}}{P_{0}}\\
& =\frac{2\left[  2x\left(  x^{2}-y^{2}+2xyi\right)  +\left(  r^{2}+24\right)
\left(  2x+2yi\right)  \right]  P_{0}}{2(P_0)^{3/2}}\\
& -\frac{\left(  r^{2}+24\right)  \left(  x^{2}-y^{2}+2xyi\right)  \left(
4r^{6}+192r^{4}+2304r^{2}+9216\right)  \left(  2x\right)  }{2(P_0)^{3/2}},
    \end{aligned}
\end{equation*}
\begin{equation*}
    \begin{aligned}
        \partial_{y}\Psi_2 & =\frac{\partial_{y}\left[  \left(  r^{2}+24\right)  \left(
x^{2}-y^{2}+2xyi\right)  \right]  \sqrt{P_{0}}-\left[  \left(  r^{2}
+24\right)  \left(  x^{2}-y^{2}+2xyi\right)  \right]  \frac{\partial_{y}P_{0}
}{2\sqrt{P_{0}}}}{P_{0}}\\
& =\frac{2\left[  2y\left(  x^{2}-y^{2}+2xyi\right)  +\left(  r^{2}+24\right)
\left(  -2y+2xi\right)  \right]  P_{0}}{2(P_0)^{3/2}}\\
& -\frac{\left(  r^{2}+24\right)  \left(  x^{2}-y^{2}+2xyi\right)  \left(
4r^{6}+192r^{4}+2304r^{2}+9216\right)  \left(  2y\right)  }{2(P_0)^{3/2}}.
    \end{aligned}
\end{equation*}
As for the kernel $i\Psi_2$, we rewrite 
$$
i\Psi_2=\frac{2P_0\left(  r^{2}+24\right)  \left(-2xy+(x^2-y^2)i\right)  }
{2(P_0)^{3/2}}
$$
Combining the above formulas with the expressions for $\eta_{\pm}$ in \eqref{eta+}-\eqref{eta-}, we can solve the following problems using the method of undetermined coefficients (since the numerators on both sides of the equations are required to be real-valued polynomials in $x$ and $y$ and all terms share a common denominator $2(P_0)^{3/2}$):
\begin{equation*}
    \text{Re}\partial _{\alpha}\Psi _{2,\alpha}|_{\alpha =0}=\text{Re}\left(c_{11} \partial_x \Psi_2+c_{21} \partial_y \Psi_2+c_{31} i\Psi_2+c_{41} \eta_+ +c_{51}\eta_-\right),
\end{equation*}
\begin{equation*}
    \text{Im}\partial _{\alpha}\Psi _{2,\alpha}|_{\alpha =0}=\text{Im}\left(c_{12} \partial_x \Psi_2+c_{22} \partial_y \Psi_2+c_{32} i\Psi_2+c_{42} \eta_+ +c_{52}\eta_-\right).
\end{equation*}
Assume that the polynomial functions 
\begin{equation*}
    \partial _{\alpha}P_{\alpha}|_{\alpha =0},\quad \partial _{\alpha}\beta _{1,\alpha}|_{\alpha =0},\quad \partial _{\alpha}\beta _{2,\alpha}|_{\alpha =0},\quad \partial _{\alpha}\gamma _{1,\alpha}|_{\alpha =0},\quad \partial _{\alpha}\gamma _{2,\alpha}|_{\alpha =0},
\end{equation*}
to be determined are of relatively low degree. Substituting their expansions into the above equations, we obtain 
\begin{equation*}
    \begin{aligned}
        \mathcal{P}_1=\partial _{\alpha}P_{\alpha}|_{\alpha =0}=&-48x^5+144x^4y+96x^3y^2-768x^3+96x^2y^3\\
&+2304x^2y+144xy^4+2304xy^2-48y^5-768y^3,
    \end{aligned}
\end{equation*}
\begin{equation}\label{beta-gamma}
    \partial _{\alpha}\beta _{1,\alpha}|_{\alpha =0}=x+y,\quad \partial _{\alpha}\gamma _{1,\alpha}|_{\alpha =0}=-x-y,\quad \partial _{\alpha}\beta _{2,\alpha}|_{\alpha =0}=\partial _{\alpha}\gamma _{2,\alpha}|_{\alpha =0}=-x+y,
\end{equation}
and for the undetermined coefficients on the right-hand side, we find that $c_{41}=c_{42}=-24$, while all other coefficients vanish.
\begin{remark}
    If we modify the factor in $\Psi _{2,\alpha}$ so that it corresponds to the zero configuration of system \eqref{system2}, then the above method of undetermined coefficients yields 
    $$
    \mathcal{P}_1=\partial _{\alpha}P_{\alpha}|_{\alpha =0}=48x^5-96x^3 y^2+768x^3-144 xy^4-2304 xy^2, 
    $$
    $$
\partial _{\alpha}\beta _{1,\alpha}|_{\alpha =0}=-x,\quad \partial _{\alpha}\gamma _{1,\alpha}|_{\alpha =0}=x,\quad \partial _{\alpha}\beta _{2,\alpha}|_{\alpha =0}=\partial _{\alpha}\gamma _{2,\alpha}|_{\alpha =0}=-y,
$$
and $c_{41}=c_{42}=c_{51}=c_{52}=12$, while all other coefficients on the right-hand side vanish. 
\end{remark}

Before proceeding to determine the polynomials corresponding to each order of $\alpha$ in the remaining factors of $\Psi _{2,\alpha}$, we first examine the approximate degree-2 solution $U^*$ constructed from three positively oriented degree-1 vortices and one negatively oriented degree-1 vortex. We set 
\begin{equation*}
    U^*=\overline{\Psi_1(z)}\Psi_1(z-\sqrt{2}\alpha e^{-3\pi i/4})\Psi_1(z-\sqrt{2}\alpha e^{-\pi i/12})\Psi_1(z-\sqrt{2}\alpha e^{7\pi i/12})=\frac{A_{\alpha}}{\sqrt{B_{\alpha}}}.
\end{equation*}
Substituting the expression for $\Phi_1$ and expanding in $\alpha$, we obtain
\begin{equation*}
    A_{\alpha}=x^4-y^4+2(-x+y)\alpha^3+\left[2xyr^2+2(x+y)\alpha^3\right]i,
\end{equation*} 
\begin{equation*}
    \begin{aligned}
B_{\alpha}&=8\alpha ^6\left( r^2+4 \right) +48\alpha ^4\left( r^2+4 \right) +24\alpha ^2\left( r^2+4 \right) ^2+\left( r^2+4 \right) ^4\\
&-4\alpha ^3\left[ x^5-3x^4y-2x^3\left( y^2-2 \right) -2x^2y\left( y^2+6 \right) -3xy^2\left( y^2+4 \right) +y^3\left( y^2+4 \right) \right].
    \end{aligned}
\end{equation*}
We expect that, for sufficiently large $\alpha$, the one-parameter family of solutions $\Psi _{2,\alpha}$ we seek will be close to the approximate solution $U^*$. In particular, we require that the coefficient polynomial of the $\alpha^3$ term in the numerator of \eqref{degree2} coincides with that of the $\alpha^3$ term in $A_{\alpha}$; and the denominator of \eqref{degree2}, $P_{\alpha}$, has its $\alpha^6$ coefficient polynomial matching exactly that of the $\alpha^6$ term in $B_{\alpha}$.
Hence, by \eqref{beta-gamma}, it is natural to impose that $\beta_{1,\alpha},\beta_{2,\alpha},\gamma_{1,\alpha},\gamma_{2,\alpha}$ as in \eqref{bb}-\eqref{gg} and $\mathcal{P}_6=8(r^2+4)$. 
In fact, under this choice, a direct computation shows that the numerator of $\Psi _{2,\alpha}$ differs from $A_{\alpha}$ only by an additional term
\begin{equation*}
    24(x^2-y^2-\alpha x+\alpha y)+24(2xy+\alpha x+\alpha y)i.
\end{equation*}
At the same time, it is precisely the additional $2\alpha^2$ term introduced in $\beta_{1,\alpha}$ and $\gamma_{1,\alpha}$ that ensures the zeros of $\Psi _{2,\alpha}$ retain an equilateral triangular configuration as the parameter $\alpha$ varies. Without this term, or if the coefficient of the $\alpha^2$ term is too small, numerical computations indicate that, for sufficiently large $\alpha$, $\Psi _{2,\alpha}$ develops ten zeros, thereby destroying the equilateral triangular structure of the zero configuration. 

At this stage, it suffices to determine the polynomials corresponding to the $\alpha$-terms of orders two through five in $P_{\alpha}$, i.e. $\mathcal{P}_2,\cdots,\mathcal{P}_5$. To this end, we consider the system formed by the sum of the squares of the real and imaginary parts of $\Psi _{2,\alpha}$, in which each term is a rational function. More precisely, we set 
$$
\Psi _{2,\alpha}=u+iv,\quad g=u^2,\quad h=v^2.
$$
Note that 
$$
\nabla u=\frac{\nabla g}{2\sqrt{g}},\quad \nabla v=\frac{\nabla h}{2\sqrt{h}},\quad 
$$
$$
\Delta u=\frac{\Delta g}{2\sqrt{g}}-\frac{1}{4g^{3/2}}|\nabla g|^2,\quad \Delta v=\frac{\Delta h}{2\sqrt{h}}-\frac{1}{4h^{3/2}}|\nabla h|^2.
$$
Substitute these into Eq.~\eqref{csg2}, and consider the real and imaginary parts separately to obtain
\begin{equation}\label{ration1}
    \begin{aligned}
4gh&\left( 2-g-h \right) \Delta g-2h\left( 2-g-h \right) |\nabla g|^2+2g\left( |\nabla g|^2h-|\nabla h|^2g \right)\\
&+4gh\nabla g \nabla h  +4g^2h\left( 1-g-h \right) \left( 2-g-h \right) ^2=0,
    \end{aligned}
\end{equation}
\begin{equation}\label{ration2}
    \begin{aligned}
4gh&\left( 2-g-h \right) \Delta h-2g\left( 2-g-h \right) |\nabla h|^2-2h\left( |\nabla g|^2h-|\nabla h|^2g \right) \\
&+4gh\nabla g \nabla h  +4gh^2\left( 1-g-h \right) \left( 2-g-h \right) ^2=0.
    \end{aligned}
\end{equation}
Therefore, the imaginary-part equation is obtained from the real-part equation simply by interchanging $g$ and $h$. Based on the previous discussion, we further express $g$ and $h$ in the following rational function form
$$
g=\frac{G}{P},\quad h=\frac{H}{P},
$$
where
$$
G=\left[ \beta _{1,\alpha}\left( x^2-y^2-\alpha x+\alpha y \right) +\beta _{2,\alpha}\left( 2xy+\alpha x+\alpha y \right) \right] ^2,
$$
$$
H=\left[ \gamma _{1,\alpha}\left( 2xy+\alpha x+\alpha y \right) +\gamma _{2,\alpha}\left( x^2-y^2-\alpha x+\alpha y \right) \right] ^2,
$$
and $P=P_{\alpha}$. For simplicity, we omit the subscript $\alpha$. Note that
$$
\partial _xg=\frac{\partial _xGP-G\partial _xP}{P^2},\quad \partial^2_{x}  g=\frac{\left( \partial^2_{x}  GP-G\partial^2_{x}  P \right) P-2\left( \partial _xGP-G\partial _xP \right) \partial _xP}{P^3}.
$$
Similarly, we can compute $\partial_y g, \partial^2_{y}  g$ and the corresponding derivatives of $h$.
Substituting these rational forms into Eq.~\eqref{ration1}, we find that all terms share a common denominator $P^6$. Multiplying both sides of the equation by $P^6$, we find Eq.~\eqref{ration1} can be transformed into the following form: 
\begin{equation}\label{ration3}
\begin{aligned}
&\quad 4GH\left( 2P-G-H \right) \left[ \left( \partial^2_{x}   G P-G\partial^2_{x}   P  \right) P-2\left( \partial_x G P-G\partial_x P \right) \partial_x P \right. \\
&\qquad\left. +\left( \partial^2_{y}   G P-G\partial^2_{y}   P  \right) P-2\left( \partial_y G P-G\partial_y P  \right) \partial_y P  \right] \\
&\quad+\left( -4HP+4GH+2H^2 \right) \left[ \left( \partial_x G P-G\partial_x P \right) ^2+\left( \partial_y G P-G\partial_y P  \right) ^2 \right] \\
&\quad-2G^2\left[ \left( \partial_x H P-H\partial_x P \right) ^2+\left( \partial_y H P-H\partial_y P  \right) ^2 \right] \\
&\quad+4GH\left[ \left( \partial_x G P-G\partial_x P \right) \left( \partial_x H P-H\partial_x P \right) +\left( \partial_y G P-G\partial_y P  \right) \left( \partial_y H P-H\partial_y P  \right) \right] \\
&\quad+4G^2H\left( P-G-H \right) \left( 2P-G-H \right) ^2=0.
\end{aligned}
\end{equation}
Interchanging $G$ and $H$ in the above equation yields the reduced form of Eq.~\eqref{ration2}.
We have now transformed the real and imaginary parts into a system of equations expressed as products of polynomials, where $G$ and $H$ are known polynomials, and $P$ is the polynomial to be determined. It is worth noting that the coefficient polynomials $\mathcal{P}_0$, $\mathcal{P}_1$ and $\mathcal{P}_6$ have already been fixed.
By requiring that the coefficient polynomials of each order in $\alpha$ vanish, we can successively determine $\{\mathcal{P}_j\}$ from $\mathcal{P}_5$ to $\mathcal{P}_2$ using the method of undetermined coefficients. In fact, in the form of Eq.~\eqref{ration3}, when we require the coefficient polynomials to vanish successively from the highest powers of $\alpha$ down to the lower ones, we avoid nonlinear coupling among the undetermined parameters, making the method of undetermined coefficients feasible in principle.
A useful observation is that the coefficient polynomial of the $O(\alpha^{35})$ term on the left-hand side of Eq.~\eqref{ration3} depends only on $\mathcal{P}_5$. We may therefore set 
$\mathcal{P}_5=0$ directly. Next, by requiring the coefficient polynomials of the $O(\alpha^k)$ terms, with $k=34,33,32$, to vanish successively, we can determine 
$\mathcal{P}_4,\mathcal{P}_3,\mathcal{P}_2$
in turn. Substituting the resulting polynomial 
$P$ back into the left-hand side of Eq.~\eqref{ration3}, one then finds that all remaining coefficients in the expansion in powers of 
$\alpha$ vanish.
This constitutes our main algorithm. 
Thus, we have obtained the one-parameter family of solutions $\Psi_{2,\alpha}$ stated in Theorem \ref{thm}. In the next section, we will further verify, using Hirota's bilinear derivatives, that $\Psi_{2,\alpha}$ indeed satisfies the CSG2 equation.

\bigskip
\section{The Hirota bilinear derivative method}\label{section4}
\subsection{The Getmanov bilinear system}\label{subsection4.1}

Let us consider the new function $\eta=\psi^2$, where $\psi$ solves CSG2. Then we compute $\nabla \eta=2\psi\nabla\psi$ and
$$
\begin{aligned}
\Delta \eta & =2 \psi \Delta \psi+2(\nabla \psi)^2 \\
& =-2 \psi\left[\frac{\bar{\psi}(\nabla \psi)^2}{2-|\psi|^2}+\frac{1}{2} \psi\left(1-|\psi|^2\right)\left(2-|\psi|^2\right)\right]+2(\nabla \psi)^2 \\
& =-\frac{1}{2} \frac{|\eta|(\nabla \eta)^2}{(2-|\eta|) \eta}+\frac{1}{2} \frac{(\nabla \eta)^2}{\eta}-\eta(1-|\eta|)(2-|\eta|) \\
& =\frac{(\nabla \eta)^2}{\eta} \frac{1-|\eta|}{2-|\eta|}-\eta(1-|\eta|)(2-|\eta|) .
\end{aligned}
$$
Therefore, the function $\eta$ satisfies the following equation 
\begin{equation}\label{eta}
    E(\eta):=\Delta \eta-\frac{\bar{\eta}(\nabla \eta)^2}{|\eta|^2} \frac{1-|\eta|}{2-|\eta|}+\eta(1-|\eta|)(2-|\eta|)=0.
\end{equation}
Replacing $\eta$ with $2\xi$ in the above equation, we recover Eq.~(3.2) in \cite{G80}, where $m^2=1/2$. 

For any complex-valued functions $f(z)=f(x,y),g(z)=g(x,y)$, let us introduce the Hirota bilinear operator D defined by 
$$
Df\cdot g=\nabla fg-f\nabla g,\quad D^2 f\cdot g=\Delta fg-2\nabla f\nabla g+f\Delta g.
$$
Now we assume that $\eta=g/f$, where $g$ is complex-valued and $f$ is a real-valued positive function. We compute 
\begin{equation*}
    \begin{aligned}
f^3\Delta \left( \frac{g}{f} \right) &=f^3\frac{f^2\left( f\Delta g-g\Delta f \right) -2f^2\nabla f\nabla g+2fg|\nabla f|^2}{f^4}\\
&=f^2\Delta g-fg\Delta f-2f\nabla f\nabla g+2g|\nabla f|^2\\
&=fD^2g\cdot f-gD^2f\cdot f,
    \end{aligned}
\end{equation*} 
\begin{equation*}
    \begin{aligned}
f^4\left( \nabla \left( \frac{g}{f} \right) \right) ^2&=f^2(\nabla g)^2+g^2(\nabla f)^2-2fg\nabla f\nabla g\\
&=\frac{g}{2}\left[ -fg\Delta f-2f\nabla f\nabla g+f^2\Delta g+2g(\nabla f)^2 \right] \\
&+\frac{f}{2}\left[ -fg\Delta g-2g\nabla g\nabla f+g^2\Delta f+2f(\nabla g)^2 \right] \\
&=\frac{g}{2}\left[ fD^2g\cdot f-gD^2f\cdot f \right] +\frac{f}{2}\left[ gD^2f\cdot g-fD^2g\cdot g \right] \\
&=gfD^2g\cdot f-\frac{1}{2}g^2D^2f\cdot f-\frac{1}{2}f^2D^2g\cdot g.
    \end{aligned}
\end{equation*}
Multiplying both sides of Eq.~\eqref{eta} by $f^3$
and substituting the above expressions, we obtain
\begin{equation}\label{LHS}
    \begin{aligned}
        \text{LHS}&=\left[ -gD^2f\cdot f+\frac{g}{2}\frac{f-|g|}{2f-|g|}D^2f\cdot f+g\left( f-|g| \right) \left( 2f-|g| \right) \right]\\
&+\left[ fD^2g\cdot f-f\frac{f-|g|}{2f-|g|}D^2g\cdot f+\frac{1}{2}\frac{\bar{g}}{|g|^2}\frac{f-|g|}{2f-|g|}f^2D^2g\cdot g \right]\\
&=\frac{g\left( |g|-3f \right)}{2\left( 2f-|g| \right)}\left[ D^2f\cdot f-2|g|\left( |g|-2f \right) \right]\\
&+\frac{\bar{g}}{|g|^2}\frac{f^2}{2f-|g|}\left[ gD^2g\cdot f+\frac{1}{2}\left( f-|g| \right) D^2g\cdot g+4g^2f-2g^2|g| \right] \\
&-g\left( |g|-3f \right) |g|+g\left( f-|g| \right) \left( 2f-|g| \right) +\frac{2f^2g|g|}{2f-|g|}-\frac{4gf^3}{2f-|g|}.
    \end{aligned}
\end{equation}
It is easy to verify that the last line on the right-hand side of the above equality vanishes. Therefore, in order to find $\eta$ satisfying Eq.~\eqref{eta}, it suffices to solve
\begin{equation}\label{system3}
    \left\{ \begin{array}{l}
	D^2f\cdot f=2|g|\left( |g|-2f \right) ,\\[0.5em]
	g\left( D^2+4 \right) g\cdot f=2g^2|g|-\frac{1}{2}\left( f-|g| \right) D^2g\cdot g.\\
\end{array} \right. 
\end{equation}
Note that the second equation in the above system is trilinear in the functions $g$ and $f$. However, we can introduce a new complex-valued function $h$ satisfying the following relation
$$
D^2g\cdot g=2gh,
$$
thereby obtaining a system of equations in the following bilinear form
\begin{equation}\label{system4}
    \left\{ \begin{array}{l}
	D^2f\cdot f=2|g|\left( |g|-2f \right) ,\\[0.5em]
	D^2g\cdot g=2gh,\\[0.5em]
	\left( D^2+4 \right) g\cdot f=2g|g|-h\left( f-|g| \right) .\\
\end{array} \right. 
\end{equation}
This is the system proposed by Getmanov \cite{G77, G80}.
Note that the reduced system above provides only a sufficient condition for a solution to the corresponding CSG2 equation. As in \cite[Section 5]{G80}, one may expand the functions $g,f,h$ into power series in the parameter $\alpha$, thereby obtaining the so-called soliton solutions of the CSG2 equation by \eqref{system4}. Unfortunately, it appears not possible to recover the previously obtained one-parameter family of solutions $\Psi _{2,\alpha}$ from this bilinear formulation, since it does not satisfy system \eqref{system3} (this can be readily verified using \textit{Mathematica}), where we take 

\begin{equation*}
    \begin{aligned}
g&=\Bigg\{
\left[ \beta _{1,\alpha}\left( x^2-y^2-\alpha x+\alpha y \right)
+\beta _{2,\alpha}\left( 2xy+\alpha x+\alpha y \right) \right] \\
&\quad + i\left[
\gamma _{1,\alpha}\left( 2xy+\alpha x+\alpha y \right)
+\gamma _{2,\alpha}\left( x^2-y^2-\alpha x+\alpha y \right) \right]
\Bigg\}^2,\\
f&=P_{\alpha}=\sum_{k=0}^6{\alpha ^k\mathcal{P}_k}.
    \end{aligned}
\end{equation*}
Meanwhile, substituting the above $g$ and $f$ into \eqref{LHS}, one readily verifies that it vanishes. 

\subsection{A modified bilinear structure}

It turns out that our solutions satisfy a new bilinear system different from the above mentioned Getmanov system. Let us explain this in more details.

We first point out that explicit expression of our solution suggests that it will be more convenient to work with the $z,\bar z$ coordinates. Note that 
$$
D_z f\cdot g=f_zg-fg_z,\quad 
D_{\bar z}f\cdot g=f_{\bar z}g-fg_{\bar z},
$$
and the relation to the first-order Hirota derivatives in real variables is given by 
$$
D_z=\frac{1}{2}(D_x-iD_y),\quad D_{\bar z}=\frac{1}{2}(D_x+iD_y),
$$
where 
$$
D_xf\cdot g=f_xg-fg_x,\quad D_yf\cdot g=f_yg-fg_y.
$$
We also have 
\[
    D^2 f\cdot g
    =4\left(
    f_{z\bar z}g+f g_{z\bar z}
    -f_zg_{\bar z}-f_{\bar z}g_z
    \right).
\]

For $\eta=\psi^2=g/f$, we propose the following new system of bilinear system: 
\begin{equation}\label{system5}
    \left\{
    \begin{aligned}
        D^2 f\cdot f
        &=2(f-|g|)^2,\\[0.5em]
        D^2 g\cdot g
        &=2gh,\\[0.5em]
        D^2 g\cdot f
        &=-(f-|g|)(g+h).
    \end{aligned}
    \right.
\end{equation}
Combining \eqref{LHS} and \eqref{system5},  a direct algebraic reduction gives
\begin{equation*}
    \begin{aligned}
        f^3E(\eta)=&\frac{g(|g|-3f)}{2(2f-|g|)}\left[2(f-|g|)^2-2|g|(|g|-2f)\right]\\
        &+\frac{\bar g f^2}{|g|^2(2f-|g|)}\left[-g(f-|g|)(g+h)+(f-|g|)gh+4g^2f-2g^2|g|\right]\\
        =&\frac{g f^2(|g|-3f)}{2f-|g|}+\frac{g f^2(3f-|g|)}{2f-|g|}\\
        =&0,
    \end{aligned}
\end{equation*}
which implies that \(\eta=g/f\) satisfies the \(\eta\)-form of the CSG2 equation \eqref{eta}. Therefore, as long as $f$ is strictly positive, for any pair $(f,g)$ satisfying the bilinear system \eqref{system5}, the corresponding function $\psi$ is a solution of the original CSG2 equation. In fact, for the multivortex solutions obtained in Theorems \ref{thm} and \ref{thm:degree3-complex-family}, taking the corresponding pair $(f,g)$, direct computations show that they indeed satisfy system \eqref{system5}. This suggests that system \eqref{system5} can provide a general algorithm for multivortex solutions of the CSG2 equation.

\bigskip
\section{Non-degeneracy and stability of vortex solutions}\label{section5}

In this section, we consider the non-degeneracy of $\Psi_1, \Psi_2$ and the stability of $\Psi_1$. A key ingredient is a precise understanding of the asymptotic behavior of the kernel of the linearized CSG2 operator at $\Phi_n$, both near the origin and at infinity. A full resolution of nondegeneracy problem for all solutions need a linearized Backlund transformation type argument.

For the asymptotic analysis, one may follow the contraction-mapping approach in \cite[Chapter 3]{PR12} (developed for Ginzburg-Landau vortex solutions), or alternatively the Frobenius series expansion method used in \cite{A03}. Notably, the latter provides a detailed asymptotic analysis of the kernels for the Ginzburg-Landau, CSG1, and CSG2 vortex solutions. Here, we directly state the corresponding results for the CSG2 case.

We first consider the homogeneous problem
\begin{equation}\label{eigen1}
\mathcal{L}_{n, m}\binom{u_m}{v_m}=0.
\end{equation}

\begin{lemma}[Section 4.2.1 in \cite{A03}]\label{asy1}
    There are four linearly independent solutions 
     $$
Z_{j,m}=\binom{Z_{j1,m}}{Z_{j2,m}},\quad j=1,2,3,4,
     $$
     to Problem \eqref{eigen1} with the following asymptotic behavior as $r\to 0$:
$$
Z_{1,m}=r^{n+m}(1+o(r))\binom{1}{1}; \quad Z_{2,m}=r^{-|n-m|}(1+o(r))\binom{1}{-1};
$$
$$
\begin{aligned}
    Z_{3,m}&=r^{-(n+m)}(1+o(r))\binom{1}{1};\\
    Z_{4,m}&=\left\{ \begin{array}{l}
	r^{|n-m|}\left( 1+o\left( r \right) \right) \left( \begin{array}{c}
	1\\
	-1\\
\end{array} \right), \quad \text{for } m\ne n,\\[1em]
	\log r\cdot \left( 1+o\left( r \right) \right) \left( \begin{array}{c}
	1\\
	-1\\
\end{array} \right), \quad \text{for } m=n.\\
\end{array} \right.
\end{aligned} 
$$
\end{lemma}

Then we consider the eigenvalue problem 
\begin{equation}\label{eigen2}
\mathcal{L}_{n, m}\binom{u_m}{v_m}=\omega^2 \binom{u_m}{v_m}.
\end{equation}

\begin{lemma}[Section 4.2.2 in \cite{A03}]\label{asy2}
    There are four linearly independent solutions 
     $$
        Y_{j,m}=\binom{Y_{j1,m}}{Y_{j2,m}},\quad j=1,2,3,4,
     $$
     to Problem \eqref{eigen2} with the following asymptotic behavior as $r\to \infty$:    
$$
\begin{aligned}
Y_{j,m}
&=\mathrm{Re}\left\{ \frac{e^{\pm i\omega r}}{\sqrt{r}}\left( \begin{array}{c}
	-\frac{4mn}{r^2}+O\left( \frac{1}{r^3} \right)\\
	1\pm \frac{4m^2-1}{8\omega r}+O\left( \frac{1}{r^2} \right)
\end{array} \right) \right\},
&& \text{for } \omega \ne 0,\ j=1,2;\\[6pt]
Y_{j,m}&=r^{\pm m}\left( \begin{array}{c}
	-\frac{4mn}{r^2}+O\left( \frac{1}{r^4} \right)\\
	1+O\left( \frac{1}{r^2} \right)
\end{array} \right),
&& \text{for } \omega =0,\ j=1,2,\ m\ne 0;\\[6pt]
Y_{1,m}
&=\log r\cdot \left( 1+O\left( \frac{1}{r^2} \right) \right)
\left( \begin{array}{c}
	0\\
	1
\end{array} \right),\\
Y_{2,m}
&=\left( 1+O\left( \frac{1}{r^2} \right) \right)
\left( \begin{array}{c}
	0\\
	1
\end{array} \right),
&& \text{for } \omega =0,\ m=0;\\[6pt]
Y_{j,m}
&=\frac{e^{\pm \sqrt{1-\omega ^2}r}}{\sqrt{r}}\left( \begin{array}{c}
	1\pm \frac{4\left( m^2+8n^2 \right) -1}{8\sqrt{1-\omega ^2}r}
	+O\left( \frac{1}{r^2} \right)\\
	\frac{4mn}{r^2}+O\left( \frac{1}{r^3} \right)
\end{array} \right),
&& \text{for } \omega ^2\ll 1,\ j=3,4.
\end{aligned}
$$
\end{lemma}
\begin{remark}
    The above asymptotic behavior of the low-frequency ($\omega^2\ll 1$) eigenfunctions is helpful for numerically analyzing the existence of kernels.
\end{remark}

Using Lemmas \ref{asy1} and \ref{asy2} on the asymptotic behavior of the kernels, we first establish the non-degeneracy of the linearized CSG2 operator at $\Phi_1$. 
\begin{lemma}[Non-degeneracy for $\Psi_1$]\label{nondegeneracy}
    For $n=1$ and any $m\in\mathbb{Z}$, the solutions of the homogeneous problem
    \begin{equation}\label{uv12}
\mathcal{L}_{1, m}\binom{u}{v}=0
\end{equation}
which are defined on all $\mathbb{R}^2$ and bounded must belong to the space 
$$
\mathrm{Span}_{\mathbb{R}}\left\{\binom{0}{\Phi_1},\ \binom{\Phi'_1}{-(1/r)\Phi_1},\ \binom{\Phi'_1}{(1/r)\Phi_1} \right\}.
$$
\end{lemma}
\begin{proof}
    By Remark \ref{rmk1}, we only need to consider the case $m\geq 0$. 

    \noindent \textbf{Step 1. }For the case $m=0$, using the fact that $(0,\Phi_1)^T\in \mathrm{Ker}\mathcal{L}_{1,0}$, we have 
    \begin{equation}\label{phi}
        \left[-\nabla^2_r+\mathfrak{B}_1\frac{d }{d r}+ \mathfrak{C}_1-\mathfrak{D}_1 \right]\Phi_1=0.
    \end{equation}
    By Lemma \ref{asy1}, there exist two constants $c_1,c_4$ such that $(0,\Phi_1)^T=c_1 Z_{1,0}+c_4 Z_{4,0}$. It remains to prove that $Z_{1,0}$ blows up at $\infty$. Then, since the space of functions in $\mathrm{Ker}\mathcal{L}_{1,0}$ that remain bounded near the origin is two-dimensional, we conclude that any bounded element of $\mathrm{Ker}\mathcal{L}_{1,0}$ must be a constant multiple of $(0,\Phi_1)^T$. 
 
    Denote $Z_{1,0}=(Z_{11,0},Z_{12,0})^T$. Then we have 
    \begin{equation}\label{Z11}
        \left[-\nabla^2_r+\mathfrak{B}_1\frac{d }{d r}+ \mathfrak{C}_1+\mathfrak{D}_1 \right]Z_{11,0}=0.
    \end{equation}
    Note that $Z_{11,0}=r(1+o(r))$ as $r\to0$. We claim that $Z_{11,0}$ is strictly positive on $(0,\infty)$. Indeed, if $Z_{11,0}(r_0)=0$ at some $r_0\in(0,\infty)$, then we can find $r_1\in(0,r_0)$ at which $Z_{11,0}$ attains a positive local maximum, that is 
    \begin{equation}\label{Z11-2}
        Z''_{11,0}(r_1)\leq 0,\quad Z'_{11,0}(r_1)=0,\quad Z_{11,0}(r_1)>0.
    \end{equation}
    By definition, it is easy to check that $\mathfrak{C}_1+\mathfrak{D}_1>0$ on $(0,\infty)$. Hence, by \eqref{Z11}-\eqref{Z11-2}, we obtain 
    $$
    \left[-\nabla^2_r+\mathfrak{B}_1\frac{d }{d r}+ \mathfrak{C}_1+\mathfrak{D}_1 \right]Z_{11,0}>0 \quad \text{at } r_1,
    $$
    which is a contradiction. 
    
    Multiplying Eq.~\eqref{phi} by $Z_{11,0}$  and Eq.~\eqref{Z11} by $\Phi_1$, and then subtracting, we obtain
    \begin{equation*}
        \begin{aligned}
            &\frac{1}{r}\frac{d }{d r}\left( r\frac{d Z_{11,0}}{d r}\Phi _1 \right) -\frac{1}{r}\frac{d }{d r}\left( r\frac{d \Phi _1}{d r}Z_{11,0} \right)\\
            &=\mathfrak{B}_1\left( \frac{d Z_{11,0}}{d r}\Phi _1-\frac{d \Phi _1}{d r}Z_{11,0} \right) +2\mathfrak{D}_1Z_{11,0}\Phi _1.
        \end{aligned}
    \end{equation*}
Since $\Phi'>0$, $Z_{11,0}>0$ and $\mathfrak{B}_1<0$, we have
\begin{equation*}
    \begin{aligned}
        \frac{d Z_{11,0}}{d r}\Phi _1&>\frac{d Z_{11,0}}{d r}\Phi _1-\frac{d \Phi _1}{d r}Z_{11,0}\\
&=\frac{1}{r}\int_0^r{t\mathfrak{B}_1\left( \frac{d Z_{11,0}}{d t}\Phi _1-\frac{d \Phi _1}{d t}Z_{11,0} \right)}d t+\frac{1}{r}\int_0^r{2t\mathfrak{D}_1Z_{11,0}\Phi _1d t}\\
&>\frac{1}{r}\int_0^r{t\mathfrak{B}_1 \frac{d Z_{11,0}}{d t}\Phi _1}d t+\frac{1}{r}\int_0^r{2t\mathfrak{D}_1Z_{11,0}\Phi _1d t}\\
&=\mathfrak{B}_1\Phi_1Z_{11,0}+\frac{1}{r}\int_{0}^{r}{tZ_{11,0}\left[2\mathfrak{D}_1 \Phi_1 -t^{-1}(t\mathfrak{B}_1\Phi_1)'\right]d t}.
    \end{aligned}
\end{equation*}
Note that 
\begin{equation}\label{asy-B1}
    \mathfrak{B}_1=-\frac{8}{r^3}+O\left( \frac{1}{r^5} \right) \quad \text{as } r\to\infty.
\end{equation}
It is easy to verify that
$$
0>\mathfrak{B}_1\Phi_1=O(r^{-3}) \quad \text{as } r\to \infty
$$
and 
$$
0<2\mathfrak{D}_1 \Phi_1 -r^{-1}(r\mathfrak{B}_1\Phi_1)'=\left\{ \begin{array}{l}
	\frac{3}{8}r+O\left( r^3 \right) \quad \text{as } r\rightarrow 0,\\[0.5em]
	1+O\left( r^{-2} \right) \quad \text{as } r\rightarrow \infty .\\
\end{array} \right. 
$$
Using the fact that $\Phi_1\to 1$ at $\infty$, we may choose $R>0$ sufficiently large such that for all $r>R$, we have 
\begin{equation}\label{Z11-3}
    \frac{d Z_{11,0}}{d r}\geq -\frac{1}{r^2}Z_{11,0}+\frac{\delta_{R}}{r}\int_{R}^{r}{t Z_{11,0}d t}, 
\end{equation}
where $\delta_R\sim1$. Setting 
$$
V(r)=\int_{R}^{r}{t Z_{11,0}d t},
$$
it follows from \eqref{Z11-3} that 
$$
\left(\frac{V'}{r}\right)'\geq -\frac{V'}{r^3}+\frac{\delta_R}{r}V\implies V''\geq \delta_R V \quad \text{for } r>R.
$$
By Gronwall's inequality, both functions $V$ and $Z_{11,0}$ grow exponentially at $\infty$. 

\noindent 
\textbf{Step 2. }Now we consider the case $m=1$. Recall that $(\Phi'_1,-(1/r)\Phi_1)^T\in\mathrm{Ker}\mathcal{L}_{1,1}$. By Lemma \ref{asy1}, there exist two constants $c_1,c_2$ ($c_2\ne 0$) such that $(\Phi'_1,-(1/r)\Phi_1)^T=c_1Z_{1,1}+c_2Z_{2,1}$. As before, we only need to prove that $Z_{1,1}=(Z_{11,1},Z_{12,1})^T$ blows up at $\infty$. Note that we have 
\begin{equation}\label{uv13}
    \left\{ \begin{array}{l}
	\left[ -\nabla _{r}^{2}+\mathfrak{B}_1\frac{d }{d r}+\frac{1}{r^2}+\mathfrak{C}_1+\mathfrak{D}_1 \right] \Phi'_1-\frac{1}{r}\mathfrak{A}_1 \Phi_1=0,\\[0.5em]
	\left[ -\nabla _{r}^{2}+\mathfrak{B}_1\frac{d }{d r}+\frac{1}{r^2}+\mathfrak{C}_1-\mathfrak{D}_1 \right]\left(-\frac{1}{r}\Phi_1\right)+\mathfrak{A}_1\Phi'_1=0,\\
\end{array} \right. 
\end{equation}  
and
\begin{equation}\label{uv14}
    \left\{ \begin{array}{l}
	\left[ -\nabla _{r}^{2}+\mathfrak{B}_1\frac{d }{d r}+\frac{1}{r^2}+\mathfrak{C}_1+\mathfrak{D}_1 \right] Z_{11,1}+\mathfrak{A}_1Z_{12,1}=0,\\[0.5em]
	\left[ -\nabla _{r}^{2}+\mathfrak{B}_1\frac{d }{d r}+\frac{1}{r^2}+\mathfrak{C}_1-\mathfrak{D}_1 \right] Z_{12,1}+\mathfrak{A}_1Z_{11,1}=0.\\
\end{array} \right. 
\end{equation}  

By Lemma \ref{asy1}, $Z_{11,1},Z_{12,1}=r^2(1+o(r))$ as $r\to 0$. We claim that both $Z_{11,1}$ and $Z_{12,1}$ are strictly positive on $(0,\infty)$. The proof proceeds analogously to the case $m=0$ in Step 1. In fact, if $Z_{11,1}(r_0)=0$ at some $r_0\in(0,\infty)$ and $Z_{11,1},Z_{12,1}>0$ on $(0,r_0)$, then there exists $r_1\in(0,r_0)$ such that 
\begin{equation}\label{Z}
    Z''_{11,1}(r_1)\leq 0,\quad Z'_{11,1}(r_1)=0,\quad Z_{11,1}(r_1)>0,\quad Z_{12,1}(r_1)>0.
\end{equation}
Combining \eqref{Z} with the fact that $r^{-2}+\mathfrak{C}_1\pm\mathfrak{D}_1>0$ and $\mathfrak{A}_1>0$, we obtain 
$$
\left[ -\nabla _{r}^{2}+\mathfrak{B}_1\frac{d }{d r}+\frac{1}{r^2}+\mathfrak{C}_1+\mathfrak{D}_1 \right] Z_{11,1}+\mathfrak{A}_1Z_{12,1}>0\quad \text{at } r_1,
$$
which is a contradiction. Similarly, for the case $Z_{12,1}(r_0)=0$ at some $r_0\in(0,\infty)$ and $Z_{11,1},Z_{12,1}>0$ on $(0,r_0)$, we can find some $r_1\in(0,r_0)$ such that 
$$
\left[ -\nabla _{r}^{2}+\mathfrak{B}_1\frac{d }{d r}+\frac{1}{r^2}+\mathfrak{C}_1-\mathfrak{D}_1 \right] Z_{12,1}+\mathfrak{A}_1Z_{11,1}>0\quad \text{at } r_1,
$$
which is also a contradiction.   

Multiplying the first equation of \eqref{uv13} by $Z_{11,1}$, the first equation of \eqref{uv14} by $\Phi'_1$, and subtracting, we obtain
\begin{equation*}
    \begin{aligned}
&\frac{1}{r}\frac{d }{d r}\left( r\frac{d Z_{11,1}}{d r}\Phi'_1 \right) -\frac{1}{r}\frac{d }{d r}\left( r\frac{d \Phi'_1}{d r}Z_{11,1} \right) \\
&=\mathfrak{B}_1\left( \frac{d Z_{11,1}}{d r}\Phi'_1-\frac{d \Phi'_1}{d r}Z_{11,1} \right) +\frac{1}{r}\mathfrak{A}_1\Phi _1Z_{11,1}+\mathfrak{A}_1Z_{12,1}\Phi'_1.
    \end{aligned}
\end{equation*}
Hence, 
\begin{equation*}
    \begin{aligned}
&r\left( \Phi'_1 \right) ^2\frac{d }{d r}\left( \frac{Z_{11,1}}{\Phi'_1} \right)\\
&=r\frac{d Z_{11,1}}{d r}\Phi'_1-r\frac{d \Phi'_1}{d r}Z_{11,1}\\
&=\int_0^r{t\mathfrak{B}_1\left( \Phi'_1 \right) ^2\frac{d }{d t}\left( \frac{Z_{11,1}}{\Phi'_1} \right) d t}+\int_0^r{\left( \mathfrak{A}_1\Phi _1Z_{11,1}+t\mathfrak{A}_1Z_{12,1}\Phi'_1 \right) d t}\\
&=r\mathfrak{B}_1\Phi'_1Z_{11,1}+\int_0^r{\left[ \mathfrak{A}_1\Phi _1-\left( \Phi'_1 \right) ^{-1}\frac{d }{d t}\left( t\mathfrak{B}_1\left( \Phi'_1 \right) ^2 \right) \right]}Z_{11,1}d t+\int_0^r{t\mathfrak{A}_1Z_{12,1}\Phi'_1d t}.
    \end{aligned}
\end{equation*}
A direct calculation shows that the integrands in the final two integrals are strictly positive on $(0,\infty)$. Therefore, there exists a constant $C_1>0$ such that
\begin{equation}\label{Z11-4}
    \frac{d }{d r}\left( \frac{Z_{11,1}}{\Phi'_1} \right)\geq \mathfrak{B}_1 \frac{Z_{11,1}}{\Phi'_1}+C_1 r^5\quad \text{for } r>R, 
\end{equation}  
where $R>0$ is sufficiently large and we use the asymptotic estimate for $\Phi'_1$ in \eqref{asy}. It follows from \eqref{asy-B1} that 
\begin{equation}\label{exp}
    1\leq \exp\left(-\int_{R}^{r}\mathfrak{B}_1d t \right)\leq C_R\quad \text{for } r>R,
\end{equation}
where the constant $C_R>0$ depends on $R$.  
Combining  \eqref{Z11-4} and \eqref{exp}, we obtain 
$$
\frac{d }{d r}\left[ \exp \left( -\int_R^r{\mathfrak{B}_1d t} \right) \cdot \frac{Z_{11,1}}{\Phi'_{1}} \right] \geq C_1 r^5\exp \left( -\int_R^r{\mathfrak{B}_1d t} \right)\geq C_1 r^5\quad\text{for } r>R.
$$
Integrating the above inequality from $R$ to $r$ yields:
$$
Z_{11,1}\gtrsim C_1 C_R^{-1}\Phi'_1 r^6 \gtrsim r^3 \quad \text{for } r>R,
$$
which implies that both $Z_{11,1}$ and $Z_{12,1}$ blow up exponentially at $\infty$ by Lemma \ref{asy2}.

\noindent
\textbf{Step 3. }Finally we treat the case $m\geq 2$. We set $\tilde{u}=(u+v)/2,\quad \tilde{v}=(u-v)/2$ and then Problem \eqref{uv12} can be rewritten as 
\begin{equation*}
    \left\{ \begin{array}{l}
	\left[ -\nabla _{r}^{2}+\mathfrak{B}_1\frac{d }{d r}+\frac{m^2}{r^2}+\mathfrak{C}_1+m\mathfrak{A}_1 \right] \tilde{u}+\mathfrak{D}_1 \tilde{v}=0,\\[0.5em]
	\left[ -\nabla _{r}^{2}+\mathfrak{B}_1\frac{d }{d r}+\frac{m^2}{r^2}+\mathfrak{C}_1-m\mathfrak{A}_1 \right] \tilde{v}+\mathfrak{D}_1 \tilde{u}=0.\\
\end{array} \right. 
\end{equation*} 
Combining the boundedness of $(u,v)$ with Lemmas \ref{asy1} and \ref{asy2}, we obtain that both $u$ and $v$ are at least of order $O(r^{m-1})$ as $r\to 0$ and $O(r^{-m})$ as $r\to\infty$ and hence belong to the space $H^1(\mathbb{R}^+,rd r)$. Now we define 
\begin{equation*}
    \begin{aligned}
\mathcal{Q}_m\left( \tilde{u},\tilde{v} \right) :=&\int_0^{\infty}{\left[ \left(\frac{d \tilde{u}}{d r}\right)^2+\mathfrak{B}_1\frac{d \tilde{u}}{d r}\tilde{u}+\left( \frac{m^2}{r^2}+\mathfrak{C}_1+m\mathfrak{A}_1 \right) \tilde{u}^2 \right] rd r}\\
&+\int_0^{\infty}{\left[ \left(\frac{d \tilde{v}}{d r}\right)^2+\mathfrak{B}_1\frac{d \tilde{v}}{d r}\tilde{v}+\left( \frac{m^2}{r^2}+\mathfrak{C}_1-m\mathfrak{A}_1 \right) \tilde{v}^2 \right] rd r}+\int_0^{\infty}{2\mathfrak{D}_1\tilde{u}\tilde{v}rd r}.
    \end{aligned}
\end{equation*}
By definition, we denote by
\begin{equation}\label{EF}
    \begin{aligned}
\mathfrak{E}(r)&=\frac{m^2}{r^2}+\mathfrak{C}_1\pm m\mathfrak{A}_1-\frac{\left( m\pm 1 \right) ^2}{r^2}\mp \frac{2m}{8+r^2}\\
&=\frac{-192+64r^2+22r^4+r^6}{2\left( 4+r^2 \right) ^2\left( 8+r^2 \right)},\\
\mathfrak{F}\left( \tilde{u},\tilde{v} \right) &=\mathfrak{E}\tilde{u}^2+\mathfrak{E}\tilde{v}^2+2\mathfrak{D}_1\tilde{u}\tilde{v}-\frac{1}{2}\Phi _{1}^{2}\left( \tilde{u}+\tilde{v} \right) ^2\\
&=\left( \mathfrak{E}-\frac{1}{2}\Phi _{1}^{2} \right) \tilde{u}^2+\left( \mathfrak{E}-\frac{1}{2}\Phi _{1}^{2} \right) \tilde{v}^2+\left( 2\mathfrak{D}_1-\Phi _{1}^{2} \right) \tilde{u}\tilde{v}.
    \end{aligned}
\end{equation}
It is worth noting that $\mathfrak{E}$ is independent of $m$ and it arises simply from expanding the first three terms on the right-hand side of the first formula with respect to $m$.
Then 
\begin{equation*}
    \begin{aligned}
\mathcal{Q}_m\left( \tilde{u},\tilde{v} \right) &=\int_0^{\infty}{\left[ \left( \frac{d \tilde{u}}{d r} \right) ^2+\left( \frac{d \tilde{v}}{d r} \right) ^2 -\frac{1}{2}r^{-1}(r\mathfrak{B}_1)'(\tilde{u}^2+\tilde{v}^2)    \right.}\\
&\quad\left. +\left( \frac{\left( m+1 \right) ^2}{r^2}+\frac{2m}{8+r^2} \right) \tilde{u}^2+\left( \frac{\left( m-1 \right) ^2}{r^2}-\frac{2m}{8+r^2} \right) \tilde{v}^2 \right] rd r\\
&\quad+\frac{1}{2}\int_0^{\infty}{\Phi _{1}^{2}\left( \tilde{u}+\tilde{v} \right) ^2rd r}+\int_0^{\infty}{\mathfrak{F}\left( \tilde{u},\tilde{v} \right) rd r}.
    \end{aligned}
\end{equation*}
Let us introduce two undetermined real-valued functions $\gamma_1(r),\gamma_2(r)$ (to be chosen later) and rewrite the integrands in the last two integrals as follows:
\begin{equation*}
    \begin{aligned}
\frac{1}{2}\Phi _{1}^{2}\left( \tilde{u}+\tilde{v} \right) ^2+\mathfrak{F}\left( \tilde{u},\tilde{v} \right) &=\frac{1}{2}\left[ \left( 1-\gamma _1 \right) \tilde{u}+\left( 1-\gamma _2 \right) \tilde{v} \right] ^2\Phi _{1}^{2}\\
&\quad+\left[ \mathfrak{E}-\frac{1}{2}\Phi _{1}^{2}-\frac{1}{2}\left( \gamma _{1}^{2}-2\gamma _1 \right) \Phi _{1}^{2} \right] \tilde{u}^2\\
&\quad+\left[ \mathfrak{E}-\frac{1}{2}\Phi _{1}^{2}-\frac{1}{2}\left( \gamma _{2}^{2}-2\gamma _2 \right) \Phi _{1}^{2} \right] \tilde{v}^2\\
&\quad+\left[ \left( \gamma _1+\gamma _2-\gamma _1\gamma _2 \right) \Phi _{1}^{2}+2\mathfrak{D}_1-\Phi _{1}^{2} \right] \tilde{u}\tilde{v}.
    \end{aligned}
\end{equation*}
We require the coefficients of the sign-indefinite term, i.e. $\tilde{u}\tilde{v}$, in the last line to vanish, which means that $\gamma_1,\gamma_2$ must satisfy the following relation:
\begin{equation}\label{gg2}
    \gamma _1\gamma _2-\gamma _1-\gamma _2=\frac{2\mathfrak{D}_1-\Phi _{1}^{2}}{\Phi _{1}^{2}}=\frac{4(20+3r^2)}{(4+r^2)(8+r^2)}.
\end{equation}
Then we estimate 
\begin{equation*}
    \begin{aligned}
        \mathcal{Q}_m\left( \tilde{u},\tilde{v} \right)&\geq \int_0^{\infty}{\left[ \frac{\left( m+1 \right) ^2}{r^2}+\frac{2m}{8+r^2}-\frac{1}{2}r^{-1}\left( r\mathfrak{B}_1 \right)'+\mathfrak{E}-\frac{1}{2}\Phi _{1}^{2}-\frac{1}{2}\left( \gamma _{1}^{2}-2\gamma _1 \right)\Phi _{1}^{2} \right] \tilde{u}^2rd r}\\
&+\int_0^{\infty}{\left[ \frac{\left( m-1 \right) ^2}{r^2}-\frac{2m}{8+r^2}-\frac{1}{2}r^{-1}\left( r\mathfrak{B}_1 \right)'+\mathfrak{E}-\frac{1}{2}\Phi _{1}^{2}-\frac{1}{2}\left( \gamma _{2}^{2}-2\gamma _2 \right)\Phi _{1}^{2} \right] \tilde{v}^2rd r}\\
&=:\int_{0}^{\infty}{\mathcal{A}(m,r) \tilde{u}^2rd r}+\int_{0}^{\infty}{\mathcal{B}(m,r) \tilde{v}^2rd r}.
    \end{aligned}
\end{equation*}
Our goal now is to find suitable $\gamma_1$ and $\gamma_2$ so that for each $m\geq 2$, both $\mathcal{A}(m,r)$ and $\mathcal{B}(m,r)$ are positive for $r\in(0,\infty)$. If this is achieved, combining with the fact that $\mathcal{Q}_m(\tilde{u},\tilde{v})=0$, we immediately obtain 
$\tilde{u}=\tilde{v}=0$. 

Note that for each fixed $r\in(0,\infty)$, when restricted to $m\geq2$, both $\mathcal{A}(m,r)$ and $\mathcal{B}(m,r)$ 
are monotonically increasing functions of $m$. 
Hence, it suffices to ensure that 
\begin{equation*}
    \mathcal{A}(2,r)>0\quad \text{and} \quad \mathcal{B}(2,r)>0\quad \text{on } (0,\infty).
\end{equation*}
Let us introduce the function $g(r)=(4+r^2)^{-1}$ and require that $\gamma_1(r)$ solves the equation
\begin{equation}\label{Ag}
    \mathcal{A}(2,r)=g(r)\quad \text{with } \gamma_{1}(r)>1\quad \text{on } (0,\infty).
\end{equation}
We compute
\begin{equation*}
    \begin{aligned}
        &\mathcal{A}(2,r)+\frac{1}{2}(\gamma^2_1-2\gamma_1)\Phi^2_1- g(r)\\
&=\frac{9}{r^2}+\frac{4}{8+r^2}-\frac{1}{2}r^{-1}\left( r\mathfrak{B}_1 \right)'+\mathfrak{E}-\frac{1}{2}\Phi _{1}^{2}- g(r)\\
&=\frac{9216+6656r^2+2096r^4+308r^6+17r^8}{r^2(4+r^2)^2(8+r^2)^2}>0,
    \end{aligned}
\end{equation*}
which implies that there exists a unique real analytic function $\gamma_1(r)$ satisfying \eqref{Ag}. Indeed, one may treat $\gamma_1$ as an unknown and use the quadratic formula to explicitly determine the value of 
$\gamma_1$ that satisfies the required condition. We remark that the choice of 
function $g(r)$ here is not unique; its sole purpose is to ensure that $\gamma_1(r)$ obtained from \eqref{Ag} satisfies the condition 
$\mathcal{A}(2,r)>0$ and that the corresponding $\gamma_2(r)$ derived from relation \eqref{gg2} satisfies the condition 
$\mathcal{B}(2,r)>0$. Below we only need to verify the latter condition. Note that 
$$
\gamma_2(r)=\frac{1}{\gamma_1(r)-1}\left[\frac{4(20+3r^2)}{(4+r^2)(8+r^2)}+\gamma_1(r)\right],
$$
which is well-defined on $(0,\infty)$. Substituting 
$\gamma_2$ into the expression for $\mathcal{B}(2,r)$, we can directly use \textit{Mathematica} to verify that  
$\mathcal{B}(2,r)>0$ with the asymptotic behavior $O(r^{-2})$ as $r\to0$ and $r\to\infty$. More precisely, we have 
$$
\mathcal{B}(2,r)=\frac{\mathcal{B}_1(2,r)}{\mathcal{B}_2(2,r)},
$$
where 
\begin{equation*}
    \begin{aligned}
        \mathcal{B}_1(2,r)&=18874368+8912896r^2+2031616r^4+1589248r^6+856576r^8+239744r^{10}\\
        &\quad +41056r^{12}+4464r^{14}+276r^{16}+7r^{18},\\
        \mathcal{B}_2(2,r)&=r^2(4+r^2)^2(8+r^2)^2(18432+13312r^2+4448r^4+744r^6+54r^8+r^{10}).
    \end{aligned}
\end{equation*}
This completes the proof of the lemma.
\end{proof}

\begin{lemma}[Almost non-degeneracy for $\Psi_2$]\label{nond2}
    For $n=2$ and any $m\in\mathbb{Z}\setminus\{\pm 2\}$, the solutions of the homogeneous problem
    \begin{equation*}
\mathcal{L}_{2, m}\binom{u}{v}=0
\end{equation*}
which are defined on all $\mathbb{R}^2$ and bounded must belong to the space 
$$
\mathrm{Span}_{\mathbb{R}}\left\{\binom{0}{\Phi_2},\ \binom{\Phi'_2}{-(2/r)\Phi_2},\ \binom{\Phi'_2}{(2/r)\Phi_2},\ \binom{u_3}{v_3},\ \binom{u_{-3}}{v_{-3}} \right\},
$$
where $(u_{\pm 3},v_{\pm 3})$ are given by Lemma \ref{m3}.
\end{lemma}

\begin{proof}
    As in the discussion of Lemma \ref{nondegeneracy}, it suffices to consider the case $m\geq 0$ by Remark \ref{rmk1}.
    The analysis of kernels for the operator $\mathcal{L}_{2,m}$ when $m=0$ and $m=1$ is analogous to Steps 1 and 2 in Lemma \ref{nondegeneracy}. In particular, by comparing with the rotational and translational kernels against $Z_{1,0}$ and $Z_{1,1}$, respectively, one shows that the latter both blow up at infinity, thereby establishing the uniqueness of bounded kernels for $m=0$ and $m=1$. 
    Now we divide the discussion into two cases: $m=3$ and $m\geq 4$. 

    \noindent \textbf{1. The case $m=3$. }By Lemma \ref{m3}, there exists a nontrivial kernel
    $$
    u_{3}=\frac{-24r\left(r^6+32 r^4+384 r^2+1536 \right)}{\left( P_0 \right) ^{3/2}}<0,\quad
    v_{3}=\frac{r}{\sqrt{P_0}}>0,
    $$
    which satisfies
\begin{equation}\label{eq-m3}
    \left\{ \begin{array}{l}
	\left[ -\nabla _{r}^{2}+\mathfrak{B}_2\frac{d }{d r}+\frac{9}{r^2}+\mathfrak{C}_2+\mathfrak{D}_2 \right]u_3 +3\mathfrak{A}_2 v_3=0,\\[0.5em]
	\left[ -\nabla _{r}^{2}+\mathfrak{B}_2\frac{d }{d r}+\frac{9}{r^2}+\mathfrak{C}_2-\mathfrak{D}_2 \right]v_3+3\mathfrak{A}_2 u_3=0.\\
\end{array} \right. 
\end{equation}
By Lemma \ref{asy1}, any element of $\mathrm{Ker}\mathcal{L}_{2,3}$ that is bounded near the origin must be a linear combination of $Z_{1,3}$ and $Z_{4,3}$. Therefore, it suffices to show that $Z_{1,3}$ blows up at infinity. This implies that the bounded kernel of $\mathcal{L}_{2,3}$ consists only of $(u_3,v_3)^T$, up to a multiplicative constant. Note that 
\begin{equation}\label{z1}
    \left\{ \begin{array}{l}
	\left[ -\nabla _{r}^{2}+\mathfrak{B}_2\frac{d }{d r}+\frac{9}{r^2}+\mathfrak{C}_2+\mathfrak{D}_2 \right] Z_{11,3}+3\mathfrak{A}_2Z_{12,3}=0,\\[0.5em]
	\left[ -\nabla _{r}^{2}+\mathfrak{B}_2\frac{d }{d r}+\frac{9}{r^2}+\mathfrak{C}_2-\mathfrak{D}_2 \right] Z_{12,3}+3\mathfrak{A}_2Z_{11,3}=0.\\
\end{array} \right. 
\end{equation}
Using the fact that $9 r^{-2}+\mathfrak{C}_2\pm\mathfrak{D}_2>0$ and $\mathfrak{A}_2>0$, 
together with the discussion in Lemma \ref{nondegeneracy}, we conclude that both $Z_{11,3}$ and $Z_{12,3}$ are strictly positive on $(0,\infty)$. 

Multiplying the second equation of \eqref{eq-m3} by $Z_{12,3}$, the second equation of \eqref{z1} by $v_3$, and subtracting, we obtain
\begin{equation*}
    \begin{aligned}
&\frac{1}{r}\frac{d }{d r}\left( r\frac{d Z_{12,3}}{d r}v_3 \right) -\frac{1}{r}\frac{d }{d r}\left( r\frac{d v_3}{d r}Z_{12,3} \right)\\
&=\mathfrak{B}_2\left( \frac{d Z_{12,3}}{d r}v_3-\frac{d v_3}{d r}Z_{12,3} \right) +3\mathfrak{A}_2v_3Z_{11,3}-3\mathfrak{A}_2u_3Z_{12,3}.
    \end{aligned}
\end{equation*}
Hence, 
\begin{equation*}
    \begin{aligned}
&r (v_3)^2 \frac{d }{d r}\left( \frac{Z_{12,3}}{v_3} \right)\\
&=r\frac{d Z_{12,3}}{d r}v_3-r\frac{d v_3}{d r}Z_{12,3}\\
&=\int_0^r{t\mathfrak{B}_2\left( v_3 \right) ^2\frac{d }{d t}\left( \frac{Z_{12,3}}{v_3} \right) d t}+\int_0^r{\left( 3t\mathfrak{A}_2 v _3 Z_{11,3}-3t\mathfrak{A}_2u_3Z_{12,3} \right) d t}\\
&=r\mathfrak{B}_2 v_3 Z_{12,3}+\int_0^r{\left[ -3 t \mathfrak{A}_2 u_3-\left( v_3 \right) ^{-1}\frac{d }{d t}\left( t\mathfrak{B}_2\left( v_3 \right) ^2 \right) \right]}Z_{12,3}d t+\int_0^r{3t\mathfrak{A}_2 v_3 Z_{11,3} d t}.
    \end{aligned}
\end{equation*}
A direct computation shows that the integrands in the last two terms on the right-hand side are both positive. Note that $\mathfrak{B}_2<0$ and  $\mathfrak{B}_2=O(r^{-3})$ at infinity, and that $v_3$ also exhibits third-order decay, an argument analogous to Step 2 in Lemma \ref{nondegeneracy} implies that $Z_{12,3}\gtrsim r^3$ for $r$ large enough.

\noindent \textbf{2. The case $m\ge 4$. }Similar to Step 3 in Lemma \ref{nondegeneracy}, let us consider the functional
\begin{equation*}
    \begin{aligned}
\mathcal{Q}_m\left( \tilde{u},\tilde{v} \right) :=&\int_0^{\infty}{\left[ \left(\frac{d \tilde{u}}{d r}\right)^2+\mathfrak{B}_2\frac{d \tilde{u}}{d r}\tilde{u}+\left( \frac{m^2}{r^2}+\mathfrak{C}_2+m\mathfrak{A}_2 \right) \tilde{u}^2 \right] rd r}\\
&+\int_0^{\infty}{\left[ \left(\frac{d \tilde{v}}{d r}\right)^2+\mathfrak{B}_2\frac{d \tilde{v}}{d r}\tilde{v}+\left( \frac{m^2}{r^2}+\mathfrak{C}_2-m\mathfrak{A}_2 \right) \tilde{v}^2 \right] rd r}+\int_0^{\infty}{2\mathfrak{D}_2\tilde{u}\tilde{v}rd r},
    \end{aligned}
\end{equation*}
where $\tilde{u}=(u+v)/2,\ \tilde{v}=(u-v)/2$. By Lemmas \ref{asy1} and \ref{asy2}, the above integrals are well-defined. We denote by 
\begin{equation*}
    \begin{aligned}
\mathfrak{E}(r)&=\frac{m^2}{r^2}+\mathfrak{C}_2\pm m\mathfrak{A}_2-\mathfrak{E}_{\pm},\\
\mathfrak{F}\left( \tilde{u},\tilde{v} \right) &=\mathfrak{E}\tilde{u}^2+\mathfrak{E}\tilde{v}^2+2\mathfrak{D}_2\tilde{u}\tilde{v}-\frac{1}{2}\Phi _{2}^{2}\left( \tilde{u}+\tilde{v} \right) ^2,\\
&=\left( \mathfrak{E}-\frac{1}{2}\Phi _{2}^{2} \right) \tilde{u}^2+\left( \mathfrak{E}-\frac{1}{2}\Phi _{2}^{2} \right) \tilde{v}^2+\left( 2\mathfrak{D}_2-\Phi _{2}^{2} \right) \tilde{u}\tilde{v},
    \end{aligned}
\end{equation*}
where 
\begin{equation}\label{Epm}
    \mathfrak{E}_{\pm}(m,r)=\frac{\left( m\pm 2 \right) ^2}{r^2}\pm \frac{4mr^2\left( r^4+48r^2+576 \right)}{r^8+80r^6+1728r^4+18432r^2+73728}.
\end{equation}
As before, the quantity $\mathfrak{E}$ here is independent of $m$. Then 
\begin{equation*}
    \begin{aligned}
\mathcal{Q}_m\left( \tilde{u},\tilde{v} \right) &=\int_0^{\infty}{\left[ \left( \frac{d \tilde{u}}{d r} \right) ^2+\left( \frac{d \tilde{v}}{d r} \right) ^2 -\frac{1}{2}r^{-1}(r\mathfrak{B}_2)'(\tilde{u}^2+\tilde{v}^2)+ \mathfrak{E}_+ (m,r)\tilde{u}^2+\mathfrak{E}_- (m,r) \tilde{v}^2 \right] rd r}\\
&\quad+\frac{1}{2}\int_0^{\infty}{\Phi _{2}^{2}\left( \tilde{u}+\tilde{v} \right) ^2rd r}+\int_0^{\infty}{\mathfrak{F}\left( \tilde{u},\tilde{v} \right) rd r}.
    \end{aligned}
\end{equation*}
Similar to \eqref{gg2}, let us introduce two undetermined real-valued functions $\gamma_{1}(r),\gamma_2(r)$ such that 
\begin{equation*}
    \gamma _1\gamma _2-\gamma _1-\gamma _2=\frac{2\mathfrak{D}_2-\Phi _{2}^{2}}{\Phi _{2}^{2}}.
\end{equation*}
Hence, 
\begin{equation*}
    \begin{aligned}
        \mathcal{Q}_m\left( \tilde{u},\tilde{v} \right)&\geq \int_0^{\infty}{\left[ \mathfrak{E}_+(m,r) -\frac{1}{2}r^{-1}\left( r\mathfrak{B}_2 \right)'+\mathfrak{E}-\frac{1}{2}\Phi _{2}^{2}-\frac{1}{2}\left( \gamma _{1}^{2}-2\gamma _1 \right)\Phi _{2}^{2} \right] \tilde{u}^2rd r}\\
&+\int_0^{\infty}{\left[ \mathfrak{E}_-(m,r)-\frac{1}{2}r^{-1}\left( r\mathfrak{B}_2 \right)'+\mathfrak{E}-\frac{1}{2}\Phi _{2}^{2}-\frac{1}{2}\left( \gamma _{2}^{2}-2\gamma _2 \right)\Phi _{2}^{2} \right] \tilde{v}^2rd r}\\
&=:\int_{0}^{\infty}{\mathcal{A}(m,r) \tilde{u}^2rd r}+\int_{0}^{\infty}{\mathcal{B}(m,r) \tilde{v}^2rd r}.
    \end{aligned}
\end{equation*}
By \eqref{Epm}, for any $r\in(0,\infty)$ and $m\geq 4$, we have 
$$
\mathcal{A}(m,r)\geq \mathcal{A}(4,r),\quad \mathcal{B}(m,r)\geq \mathcal{B}(4,r). 
$$
Hence, it suffices to ensure that 
\begin{equation*}
    \mathcal{A}(4,r)>0\quad \text{and} \quad \mathcal{B}(4,r)>0\quad \text{on } (0,\infty).
\end{equation*}
Let us choose $g(r)=(4+r^2)^{-1}$ and use the numerical software \textit{Mathematica} to determine $\gamma_1(r)>1$ and the corresponding $\gamma_2(r)$ such that for $r\in (0,\infty)$, 
$$
\mathcal{A}(4,r)=g(r)\quad \text{and}\quad  0<\mathcal{B}\left( 4,r \right) =\left\{ \begin{array}{l}
	\frac{4}{r^2}-1+O\left( r^4 \right) ,\quad \text{as\ }r\rightarrow 0,\\[0.5em]
	O\left( \frac{1}{r^2} \right) ,\quad \text{as\ }r\rightarrow \infty .\\
\end{array} \right. 
$$
which implies that $\tilde{u}=\tilde{v}=0$.
\end{proof}

\begin{remark}
    Here we exclude the case $|m|=2$ due to a technical limitation. The main difficulty is that we are unable to exploit the nonexistence of bounded kernels for $|m|\geq 2$ in the case $n=1$ and for $|m|\geq 4$ in the case $n=2$ to construct suitable parameter functions $\gamma_1(r),\gamma_2(r)$ such that, when $|m|=2$ and $n=2$, the corresponding quantities $\mathcal{A}(2,r)$ and $\mathcal{B}(2,r)$ are positive. In contrast, for the GL equation, a long-standing open problem concerns the nonexistence of kernels for degree-$n$ vortex solutions in the case $2<m<2n$. For the CSG2 equation, the existence of a nontrivial kernel in the case $m=3$ allows us to establish the uniqueness of bounded kernels at $m=3$ via comparison principles and Gronwall's inequality.
    However, for the case $m=2$, although numerical evidence strongly suggests the nonexistence of bounded kernels, we are unable—much as in the GL setting—to provide a rigorous proof using the methods developed in this paper. Consequently, the non-degeneracy of the degree-2 solution to the CSG2 equation is not fully established here.
\end{remark}

\begin{lemma}[Stability of $\Psi_1$]
    If there exists $\lambda>0$ and $\phi \in H^1_0(\mathbb{R}^2;\mathbb{C})$ such that $\mathcal{L}_1 \phi=-\lambda \phi$, then $\phi=0$.
\end{lemma}
\begin{proof}
    Assume that $\phi$ is expanded according to formula \eqref{fourier}. 
    For each integer $m\geq 0$, let us consider the following eigenvalue problem:
    $$
    \mathcal{L}_{1, m}\binom{u}{v}=-\lambda\binom{u}{v},
    $$
    where $u,v\in H^1_0(\mathbb{R}^+;\mathbb{R})$. It suffices to show that both $u$ and $v$ vanish identically. The proof follows the same strategy as Step 3 in Lemma \ref{nondegeneracy}, with the only difference being the presence of the additional $\lambda$-term. For simplicity, we present only the main steps of the proof.

    Setting $\tilde{u}=(u+v)/2,\quad \tilde{v}=(u-v)/2$, we have 
\begin{equation*}
    \left\{ \begin{array}{l}
	\left[ -\nabla _{r}^{2}+\mathfrak{B}_1\frac{d }{d r}+\frac{m^2}{r^2}+\mathfrak{C}_1+m\mathfrak{A}_1 \right] \tilde{u}+\mathfrak{D}_1 \tilde{v}=-\lambda \tilde{u},\\[0.5em]
	\left[ -\nabla _{r}^{2}+\mathfrak{B}_1\frac{d }{d r}+\frac{m^2}{r^2}+\mathfrak{C}_1-m\mathfrak{A}_1 \right] \tilde{v}+\mathfrak{D}_1 \tilde{u}=-\lambda\tilde{v}.\\
\end{array} \right. 
\end{equation*} 
Define 
\begin{equation*}
    \begin{aligned}
\mathcal{Q}_m\left( \tilde{u},\tilde{v} \right) :=&\int_0^{\infty}{\left[ \left(\frac{d \tilde{u}}{d r}\right)^2+\mathfrak{B}_1\frac{d \tilde{u}}{d r}\tilde{u}+\left( \frac{m^2}{r^2}+\mathfrak{C}_1+m\mathfrak{A}_1+\lambda \right) \tilde{u}^2 \right] rd r}\\
&+\int_0^{\infty}{\left[ \left(\frac{d \tilde{v}}{d r}\right)^2+\mathfrak{B}_1\frac{d \tilde{v}}{d r}\tilde{v}+\left( \frac{m^2}{r^2}+\mathfrak{C}_1-m\mathfrak{A}_1+\lambda \right) \tilde{v}^2 \right] rd r}+\int_0^{\infty}{2\mathfrak{D}_1\tilde{u}\tilde{v}rd r}.
    \end{aligned}
\end{equation*}
Assume that $\mathfrak{E}$, $\gamma_1$ and $\gamma_2$ satisfy formulas \eqref{EF} and \eqref{gg2}. Then we have 
\begin{equation*}
    \begin{aligned}
        \mathcal{Q}_m\left( \tilde{u},\tilde{v} \right)&\geq \int_0^{\infty}{\left[ \frac{m^2}{r^2}+\mathfrak{C}_1+m\mathfrak{A}_1+\lambda-\frac{1}{2}r^{-1}\left( r\mathfrak{B}_1 \right)'-\frac{1}{2}\Phi _{1}^{2}(\gamma_1-1)^2 \right] \tilde{u}^2rd r}\\
&+\int_0^{\infty}{\left[ \frac{m^2}{r^2}+\mathfrak{C}_1-m\mathfrak{A}_1+\lambda-\frac{1}{2}r^{-1}\left( r\mathfrak{B}_1 \right)'-\frac{1}{2}\Phi _{1}^{2}(\gamma_2-1)^2 \right] \tilde{v}^2rd r}\\
&=:\int_{0}^{\infty}{\mathcal{A}(m,r) \tilde{u}^2rd r}+\int_{0}^{\infty}{\mathcal{B}(m,r) \tilde{v}^2rd r}.
    \end{aligned}
\end{equation*}
By the discussion in Step 3 of Lemma \ref{nondegeneracy} and the relation \eqref{EF}, it suffices to consider the cases $m=0$ and $m=1$.
Our goal is to find a suitable positive function $g(m,r)$ such that
\begin{equation*}
    \mathcal{A}(m,r)=g(m,r)\quad \text{and} \quad \mathcal{B}(m,r)>0\quad \text{on } (0,\infty).
\end{equation*}
We set 
$$
f_{\pm}(m,r)=\frac{m^2}{r^2}+\mathfrak{C}_1\pm m\mathfrak{A}_1-\frac{1}{2}r^{-1}\left( r\mathfrak{B}_1 \right)'.
$$
By the condition, we have  
$$
\frac{1}{2}\Phi _{1}^{2}\left( \gamma _2-1 \right) ^2=\frac{1}{2}\Phi _{1}^{2}\left[ \frac{2\mathfrak{D}_1}{\Phi _{1}^{2}\left( \gamma _1-1 \right)} \right] ^2=\frac{\mathfrak{D}_{1}^{2}}{f_+\left( m,r \right) +\lambda -g\left( m,r \right)}.
$$
Therefore, 
$$
\mathcal{B}(m,r)=f_-(m,r)+\lambda-\frac{\mathfrak{D}_{1}^{2}}{f_+\left( m,r \right) +\lambda -g\left( m,r \right)}.
$$
For the case $m=0$, since $f_+\left( 0,r \right)>f_-\left( 0,r \right)>0$, we can choose $g(0,r)=c (f_+(0,r)+\lambda)$, where $c>1$, to ensure that $\mathcal{B}(0,r)>\lambda>0$. As for the case $m=1$, note that $f_+\left( 1,r \right)>0$ and 
$$
\min_{r\in(0,\infty)} f_-\left( 1,r \right)=-\frac{3}{4}<0.
$$
Hence, to ensure $\mathcal{B}(1,r)>\lambda>0$, it suffices to choose $g(1,r)>f_+\left( 1,r \right) +\lambda>0$ satisfying 
$$
\frac{\mathfrak{D}_{1}^{2}}{f_+\left( 1,r \right) +\lambda -g\left( 1,r \right)}=\min_{r\in(0,\infty)} f_-\left( 1,r \right)=-\frac{3}{4}.
$$
\end{proof}

\bigskip
\section{Lyapunov-Schmidt reduction analysis for solutions with large parameter}\label{section6}

In this section, we analyze the  multivortex solutions for the CSG2 equation from the perspective of the Lyapunov-Schmidt reduction framework. Formally, this analysis also tells us (more or less well known) that the vortex points, if they are far away from each other, should be determined by a pair of consecutive Adler-Moser polynomials, whose roots are approximately these vortex points, since the roots should satisfy a balancing condition.

Let us consider $(n+1)$-tuples (referred to as vortex configuration): 
\begin{equation*}
    \mathbf{q}=(q_0,q_1,\cdots,q_{n})\in \mathbb{C}^{n+1}\simeq \mathbb{R}^{2(n+1)},\quad (\sigma_0,\sigma_1,\cdots,\sigma_n)\in \{-1,1\}^{n+1}.
\end{equation*}
Let $I_{\pm}$ denote the set of indices $j$ such that $\sigma_j=\pm1$, respectively.
Recall that the degree-1 vortex solution to the CSG2 equation is given by $\Psi_1(z)=\Phi_1(r)e^{i\theta}$ and $\Psi_{-1}(z)=\overline{\Psi_1(z)}=\Phi_1(r)e^{-i\theta}$. 

From now on, we denote by $(r_j,\theta_j)$ the polar coordinates with respect to the point $q_j$ and set $z_j=z-q_j=x_j+i y_j$.
Let us define the approximate $n$-vortex solution to the CSG2 equation as follows:
\begin{equation*}
    u^*=\prod_{j=0}^{n}{\Psi_{\sigma_j}(z_j)}=\prod_{j=0}^{n}{\Phi_{1}(r_j)e^{i\sigma_j\theta_j}}.
\end{equation*}
For the solution $u$ to the CSG2 equation, we set $u=u^*+\phi$, where $\phi$ is the perturbation term. Then we can rewrite the CSG2 equation as
\begin{equation*}
\begin{aligned}
    E(u^*+\phi)&=\Delta (u^*+\phi)
+
\frac{\left[\nabla (u^*+\phi)\right]^2\left(\overline{u^*}+\bar{\phi}\right)}
{2-\left|u^*+\phi\right|^2}
\\
&
+
\frac{1}{2}(u^*+\phi)
\left(1-\left|u^*+\phi\right|^2\right)
\left(2-\left|u^*+\phi\right|^2\right)
=0.
\end{aligned}
\end{equation*}
Setting $\phi=u^* \psi$, we have 
\begin{equation}\label{linearized}
    \mathbb{L}^*\psi=(u^*)^{-1}E(u^*)+(u^*)^{-1}N(\phi),
\end{equation}
where $\mathbb{L}^*$ is the linearized operator at $u^*$ given by
\begin{equation*}
    \begin{aligned}
        \mathbb{L}^*\psi &=-\Delta \psi +|u^*|^2\left( 2-|u^*|^2 \right) \mathrm{Re}\psi +|u^*|^2\left( 1-|u^*|^2 \right) \mathrm{Re}\psi\\
&\quad -\frac{4\left( \nabla u^* \right) ^2\overline{u^*}}{\left( 2-|u^*|^2 \right) u^*}\mathrm{Re}\psi-\frac{4\nabla u^*\nabla \psi}{\left( 2-|u^*|^2 \right) u^*},
    \end{aligned}   
\end{equation*}
and $E(u^*)$ is the error term, and $N(\phi)$ is the higher-order perturbation term. Here we explain the relationship between $\mathbb{L}^*$ and the previously introduced linearized operator $\mathcal{L}_1$ (see \eqref{Ln}). In fact, if we replace $u^*$ by $\Psi_{\sigma_j}$ in the expression of $\mathbb{L}^*$ and denote the corresponding operator as $\mathbb{L}_{\sigma_j}$, then we have
\begin{equation*}
    \mathbb{L}_{\sigma_j}\psi =\Phi_1^{-1}\mathcal{L}_{\sigma_j}(\Phi_1 \psi).  
\end{equation*}
Hence, we say that $\mathbb{L}_{\sigma_j}$ and $\mathcal{L}_{\sigma_j}$ are \textbf{conjugate operators} to each other. 

We require the vortex points $\mathbf{q}$ satisfying the following configuration conditions:
\begin{itemize}
    \item $q_j=\sum_{k=1}^{\infty}\alpha^{3-2k}q_{j,k}$, where $q_{j,k}=O(1)$ for each $j,k$, and $\alpha\gg 0$.
    \item $\min\{|q_{j,1}-q_{\ell,1}|\ :\ j\neq\ell\}\sim \alpha$.
    \item The leading term of $q_j$, i.e. $q_{j,1}$, satisfies the following \textbf{balancing condition}:
    \begin{equation}\label{BC}
        \sum_{\ell\neq j}\frac{\sigma_{\ell}}{q_{j,1}-q_{\ell,1}}=0,\quad\text{for } j=0,1,\cdots,n.
    \end{equation}
\end{itemize}

For Eq.~\eqref{linearized}, we now compute the corresponding error term on the right-hand side. We shall see that, under the above vortex configuration condition, the corresponding error term is of order $O(\alpha^{-2})$ near each single vortex.

\begin{lemma}\label{Error}
    $$
    E(u^*)=u^*(E_1+i E_2),
    $$
    where the real-valued functions
    $$
    E_1=E_{11}+E_{12}+E_{13}+E_{14},\quad E_2=E_{21}+E_{22}, 
    $$
    and 
    \begin{equation*}
    \begin{aligned}
        E_{11}&=\frac{-2}{2-|u^*|^2}\sum_{j=0}^n{\sum_{\ell \ne j}{\frac{\sigma _j}{r_j}\frac{\sigma _{\ell}}{r_{\ell}}  \cos \left( \theta _j-\theta _{\ell} \right)}},\\
        E_{12}&=\frac{2}{2-|u^*|^2}\sum_{j=0}^n{\sum_{\ell \ne j}{\frac{\Phi _{1}'\left( r_j \right)}{\Phi _1\left( r_j \right)}\frac{\Phi _{1}'\left( r_{\ell} \right)}{\Phi _1\left( r_{\ell} \right)}  \cos \left( \theta _j-\theta _{\ell} \right)}},\\
E_{13}&=\sum_{j=0}^n{\left(\frac{|u^*|^2}{2-|u^*|^2}- \frac{\Phi _1\left( r_j \right) ^2}{2-\Phi_{1}(r_j)^2} \right) \left( \frac{\Phi _{1}'\left( r_j \right) ^2}{\Phi _1\left( r_j \right) ^2}-\frac{1}{r_{j}^{2}} \right)},\\
E_{14}&=\frac{1}{2}\left( 1-|u^*|^2 \right) \left( 2-|u^*|^2 \right) -\frac{1}{2}\sum_{j=0}^n{\left( 1-\Phi_{1}(r_j)^2 \right) \left( 2-\Phi_{1}(r_j)^2 \right)},\\
E_{21}&=\frac{2}{2-|u^*|^2}\sum_{j=0}^n{\sum_{\ell \ne j}{ \frac{\Phi _{1}'\left( r_j \right)}{\Phi _1\left( r_j \right)}\frac{\sigma _{\ell}}{r_{\ell}}}}\sin \left( \theta _j-\theta _{\ell} \right),\\
E_{22}&=\frac{-2}{2-|u^*|^2}\sum_{j=0}^n{\sum_{\ell \ne j}{\frac{\Phi _{1}'\left( r_{\ell} \right)}{\Phi _1\left( r_{\ell} \right)}\frac{\sigma _j}{r_j}}}\sin \left( \theta _j-\theta _{\ell} \right).
    \end{aligned}
\end{equation*}  
\end{lemma}
\begin{proof}
    For $z=x+iy=re^{i\theta}$ and $\bar{z}=x-iy=re^{-i\theta}$, we set 
    \begin{equation*}
        \begin{aligned}
            \partial&=2\partial_z=\partial_x-i\partial_y=e^{-i\theta}\left(\partial_r -\frac{i}{r}\partial_{\theta}\right),\\
            \bar{\partial}&=2\partial_{\bar{z}}=\partial_x+i\partial_y=e^{i\theta}\left(\partial_r +\frac{i}{r}\partial_{\theta}\right).
        \end{aligned}
    \end{equation*}
    The above expression remains valid if the polar coordinates are replaced by $(r_j,\theta_j)$. 
By direct computation, we have 
\begin{equation*}
    \begin{aligned}
\Delta u^*=&\partial \bar{\partial}\left( \prod_{j=0}^n{\Psi_{\sigma_j}(z_j)} \right) =\partial \left( u^*\sum_{j=0}^n{\frac{\bar{\partial}\left(\Psi_{\sigma_j}(z_j)\right)}{\Psi_{\sigma_j}(z_j)}} \right) \\
=&u^*\left( \sum_{j=0}^n{\sum_{\ell \ne j}{\frac{\bar{\partial}\left(\Psi_{\sigma_j}(z_j)\right)}{\Psi_{\sigma_j}(z_j)}\cdot \frac{\partial \left(\Psi_{\sigma_\ell}(z_\ell)\right)}{\Psi_{\sigma_\ell}(z_\ell)}}}+\sum_{j=0}^n{\frac{\Delta \left(\Psi_{\sigma_j}(z_j)\right)}{\Psi_{\sigma_j}(z_j)}} \right)\\
=&u^*\left( \sum_{j=0}^n{\sum_{\ell \ne j}{\frac{\bar{\partial}\left(\Psi_{\sigma_j}(z_j)\right)}{\Psi_{\sigma_j}(z_j)}\cdot \frac{\partial \left(\Psi_{\sigma_\ell}(z_\ell)\right)}{\Psi_{\sigma_\ell}(z_\ell)}}}\right.\\
&\quad \left.
-\sum_{j=0}^n{\frac{\left[ \nabla \left(\Psi_{\sigma_j}(z_j)\right) \right] ^2}{2-\Phi_{1}(r_j)^2}\cdot \frac{\overline{\Psi_{\sigma_j}(z_j)}}{\Psi_{\sigma_j}(z_j)}}-\frac{1}{2}\sum_{j=0}^n{\left( 1-\Phi_{1}(r_j)^2 \right) \left( 2-\Phi_{1}(r_j)^2 \right)} \right),
    \end{aligned}
\end{equation*}
and
\begin{equation*}
    \begin{aligned}
        \frac{\left( \nabla u^* \right)^2 \overline{u^*}}{2-|u^*|^2}
        &=
        \frac{\overline{u^*}}{2-|u^*|^2}
        \left(
            u^* \sum_{j=0}^n
            \frac{\partial \left(\Psi_{\sigma_j}(z_j)\right)}
            {\Psi_{\sigma_j}(z_j)}
        \right)
        \left(
            u^* \sum_{j=0}^n
            \frac{\bar{\partial}\left(\Psi_{\sigma_j}(z_j)\right)}
            {\Psi_{\sigma_j}(z_j)}
        \right) \\
        &=
        \frac{|u^*|^2 u^*}{2-|u^*|^2}
        \sum_{j,\ell=0}^n
        \frac{\bar{\partial}\left(\Psi_{\sigma_j}(z_j)\right)}
        {\Psi_{\sigma_j}(z_j)}
        \cdot
        \frac{\partial \left(\Psi_{\sigma_\ell}(z_\ell)\right)}
        {\Psi_{\sigma_\ell}(z_\ell)} .
    \end{aligned}
\end{equation*}
Hence, 
\begin{equation}\label{E1}
    \begin{aligned}
        E(u^*)&=u^*\left( \sum_{j=0}^n{\sum_{\ell \ne j}{\frac{\bar{\partial}\left(\Psi_{\sigma_j}(z_j)\right)}{\Psi_{\sigma_j}(z_j)}\cdot \frac{\partial \left(\Psi_{\sigma_\ell}(z_\ell)\right)}{\Psi_{\sigma_\ell}(z_\ell)}}}\right.\\
        &\quad+\frac{|u^*|^2}{2-|u^*|^2}
        \sum_{j,\ell=0}^n
        \frac{\bar{\partial}\left(\Psi_{\sigma_j}(z_j)\right)}
        {\Psi_{\sigma_j}(z_j)}
        \cdot
        \frac{\partial \left(\Psi_{\sigma_\ell}(z_\ell)\right)}
        {\Psi_{\sigma_\ell}(z_\ell)}-\sum_{j=0}^n{\frac{\left[ \nabla \left(\Psi_{\sigma_j}(z_j)\right) \right] ^2}{2-\Phi_{1}(r_j)^2}\cdot \frac{\overline{\Psi_{\sigma_j}(z_j)}}{\Psi_{\sigma_j}(z_j)}}\\
        &\quad+\left.\frac{1}{2}(1-|u^*|^2)(2-|u^*|)^2-\frac{1}{2}\sum_{j=0}^n{\left( 1-\Phi_{1}(r_j)^2 \right) \left( 2-\Phi_{1}(r_j)^2 \right)}\right).
    \end{aligned}
\end{equation}
Note that 
\begin{equation*}
    \begin{aligned}
        \frac{\bar{\partial}\left(\Psi_{\sigma_j}(z_j)\right)}
        {\Psi_{\sigma_j}(z_j)}&=\left( \frac{\Phi _{1}'\left( r_j \right)}{\Phi _1\left( r_j \right)}-\frac{\sigma _j}{r_j} \right) e^{i\theta _j},\\
        \frac{\partial\left(\Psi_{\sigma_j}(z_j)\right)}
        {\Psi_{\sigma_j}(z_j)}&=\left( \frac{\Phi _{1}'\left( r_j \right)}{\Phi _1\left( r_j \right)}+\frac{\sigma _j}{r_j} \right) e^{-i\theta _j},\\
        \frac{\left[ \nabla \left(\Psi_{\sigma_j}(z_j)\right) \right] ^2}{2-\Phi_{1}(r_j)^2}\cdot \frac{\overline{\Psi_{\sigma_j}(z_j)}}{\Psi_{\sigma_j}(z_j)}&=\Phi _{1}'\left( r_j \right) ^2-\frac{1}{r_{j}^{2}}\Phi _1\left( r_j \right) ^2.
    \end{aligned}
\end{equation*}
Substituting the above expressions into the error term \eqref{E1} and separating the terms in parentheses into their real and imaginary parts, we can obtain the desired expansions for $E_1$ and $E_2$. 
\end{proof}

Choose the positive constant
$$
\delta<\frac{1}{4\alpha}\min\{|q_j-q_\ell|\ :\ j\neq\ell\},
$$
and for each $j$, set 
$$
\Omega_j=\{|z-q_j|\leq \delta \alpha\}.
$$
Then for $z\in\Omega_j$ and $\ell\ne j$, 
by the balancing condition \eqref{BC}, we have
\begin{equation}\label{cos}
    \begin{aligned}
        \sum_{\ell\ne j}\frac{\sigma_\ell}{r_\ell}\cos\theta_\ell&=\sum_{\ell\ne j}\sigma_\ell\mathrm{Re}\left[\frac{z-q_j+q_j-q_\ell}{|z-q_j+q_j-q_\ell|^2}\right]\\
        &=\sum_{\ell\neq j}\sigma_\ell \mathrm{Re}\left[\frac{q_{j,1}-q_{\ell,1}}{|q_{j,1}-q_{\ell,1}|^2}\right]+O(\alpha^{-2}r_j)\\
        &=O(\alpha^{-2}r_j). 
    \end{aligned}
\end{equation}
Similarly, 
\begin{equation}\label{sin}
    \sum_{\ell\ne j}\frac{\sigma_\ell}{r_\ell}\sin\theta_\ell=O(\alpha^{-2}r_j).
\end{equation}
Hence, it follows from Lemma \ref{Error}, \eqref{cos} and \eqref{sin} that
$$
\Phi_1(r_j)E_{11}=\frac{-2\Phi_1(r_j)}{2-|u^*|^2}\sum_{k=0}^n\sum_{\ell \ne k}{\frac{\sigma _k}{r_k}\frac{\sigma _{\ell}}{r_{\ell}}  \cos \left( \theta _k-\theta _{\ell} \right)}=O(\alpha^{-2}),
$$
$$
\Phi_1(r_j)E_2=\frac{4\Phi_1(r_j)}{2-|u^*|^2}\sum_{k=0}^n\sum_{\ell \ne k}{\frac{\Phi _{1}'\left( r_k \right)}{\Phi _1\left( r_k \right)}\frac{\sigma _{\ell}}{r_{\ell}}  \sin \left( \theta _k-\theta _{\ell} \right)}=O(\alpha^{-2}).
$$
Meanwhile, $\Phi_1(r_j)E_{13}$ and $\Phi_1(r_j)E_{14}$ are of $O(\alpha^{-2})$, whereas $\Phi_1(r_j)E_{12}$ is of $O(\alpha^{-3})$.  

We now explain why the configuration condition requires the leading-order vortex locations to satisfy the balancing condition in the following lemma. Indeed, near each single vortex, the degree-1 linearized theory requires the leading term of the error projection onto the translational kernel to vanish.
\begin{lemma}
    For $z\in \Omega_j$, we have
    \begin{align*}
        \begin{aligned} 
            & \mathrm{Re}\left[\overline{\Psi_{\sigma_{j}}\left(z_{j}\right)\left(E_1+iE_2\right)}\partial_{x}\left(\Psi_{\sigma_{j}}\left(z_{j}\right)\right)\right]\\ 
            = & \frac{-4}{2-\Phi _1\left( r_j \right)^2}\Phi_1\left(r_{j}\right)\Phi'_1\left(r_{j}\right)\frac{\sigma_{j}}{r_{j}}\sum_{\ell\neq j}{\left[\alpha^{-1}\mathrm{Re}\left(\frac{\sigma_{\ell}}{q_{j,1}-q_{\ell,1}}\right)\right]}+O\left(\alpha^{-2}(1+r_{j})^{-3}\right),\\[0.5em]
            & \mathrm{Re}\left[\overline{\Psi_{\sigma_{j}}\left(z_{j}\right)\left(E_1+iE_2\right)}\partial_{y}\left(\Psi_{\sigma_{j}}\left(z_{j}\right)\right)\right]\\ 
            = & \frac{4}{2-\Phi _1\left( r_j \right)^2}\Phi_1\left(r_{j}\right)\Phi'_1\left(r_{j}\right)\frac{\sigma_{j}}{r_{j}}\sum_{\ell\neq j}{\left[\alpha^{-1}\mathrm{Im}\left(\frac{\sigma_{\ell}}{q_{j,1}-q_{\ell,1}}\right)\right]}+O\left(\alpha^{-2}(1+r_{j})^{-3}\right).
        \end{aligned}
    \end{align*}
    Therefore, if \(\{q_{j,1}\}\) satisfies the balancing condition, the \(O(\alpha^{-1})\) terms on the right-hand side of the above expressions vanish. 
\end{lemma}
\begin{proof}
    We compute only the first expansion; the second follows similarly.
    By Lemma \ref{Error} and \eqref{cos}-\eqref{sin}, we obtain
    \begin{equation*}
    \begin{aligned}
&\mathrm{Re}\left[ \overline{\Psi_{\sigma_{j}}\left(z_{j}\right)\left( E_1+i E_2 \right) }\partial _{x}\left(\Psi_{\sigma_{j}}\left(z_{j}\right)\right) \right] \\
=&\mathrm{Re}\left\{ \frac{4\Phi _1\left( r_j \right)}{2-\Phi _1\left( r_j \right)^2}\left[ -\sum_{\ell \ne j}{\frac{\sigma _j}{r_j}\frac{\sigma _{\ell}}{r_{\ell}}\cos \left( \theta _j-\theta _{\ell} \right)} \right. \right. \\
&\quad\left. \left. -i\sum_{\ell \ne j}{\frac{\Phi _{1}'\left( r_j \right)}{\Phi _1\left( r_j \right)}\frac{\sigma _{\ell}}{r_{\ell}}\sin \left( \theta _j-\theta _{\ell} \right)} \right] \cdot \left[ \Phi _{1}'\left( r_j \right) \cos \theta _j-i\frac{\sigma _j}{r_j}\Phi _1\left( r_j \right) \sin \theta _j \right] \right\}\\
&+O\left( \alpha^{-2} (1+r_j)^{-3} \right)\\
=&\frac{-4}{2-\Phi _1\left( r_j \right)^2}\Phi _1\left( r_j \right) \Phi _{1}'\left( r_j \right) \frac{\sigma _j}{r_j}\sum_{\ell \ne j}{\frac{\sigma _{\ell}}{r_{\ell}}}\cos \theta _{\ell}+O\left( \alpha^{-2} (1+r_j)^{-3} \right)\\
=&\frac{-4}{2-\Phi _1\left( r_j \right)^2}\Phi _1\left( r_j \right) \Phi _{1}'\left( r_j \right) \frac{\sigma _j}{r_j}\sum_{\ell \ne j}{\left[ \alpha^{-1} \mathrm{Re}\left(\frac{\sigma _{\ell}}{q_{j,1}-q_{\ell,1}}\right) \right]}+O\left( \alpha^{-2} (1+r_j)^{-3} \right).
    \end{aligned}
\end{equation*}
\end{proof}

To improve the approximate solution near each single vortex, we need to further compute the expansion of the error term with respect to the parameter $\alpha$. By the expansion of the error term in Lemma \ref{Error}, for distinct indices \(k,\ell,j\) and \(z\in \Omega_j\), it suffices to estimate the following two quantities:
$$
A_{j\ell}=\frac{1}{r_\ell}\cos(\theta_j-\theta_\ell)\quad \text{and}\quad B_{j\ell}=\frac{1}{r_\ell}\sin(\theta_j-\theta_\ell).
$$
All the remaining lower-order terms can be expressed in terms of the two formulas above; for example,
$$
\frac{1}{r^2_\ell}=A^2_{j\ell}+B^2_{j\ell},
$$
\begin{equation*}
    \begin{aligned}
        \frac{1}{r_k}\frac{1}{r_\ell}\cos(\theta_k-\theta_\ell)&=A_{jk}A_{j\ell}+B_{jk}B_{j\ell},\\
        \frac{1}{r_k}\frac{1}{r_\ell}\sin(\theta_k-\theta_\ell)&=A_{jk}B_{j\ell}-A_{j\ell}B_{jk}.
    \end{aligned}
\end{equation*}
Denote by 
$$
q_j=q_{j1}+i q_{j2},\quad a_{j\ell}=q_{j1}-q_{\ell1},\quad b_{j\ell}=q_{j2}-q_{\ell2},
$$
and
$$
d_{j\ell}=|q_j-q_\ell|=\sqrt{a_{j\ell}^2+b_{j\ell}^2}\sim \alpha.
$$
Note that
$$
A_{j\ell}=\frac12\frac{x_{\ell}}{r_{\ell}^2}\left(e^{i\theta_{j}}+e^{-i\theta_{j}}\right)-\frac{i}{2}\frac{y_{\ell}}{r_{\ell}^2}\left(e^{i\theta_{j}}-e^{-i\theta_{j}}\right). 
$$
We first consider only the terms of order $\alpha^{-1}$ and $\alpha^{-2}$, and formally restrict to the regime \(\{|z_j|\leq C\}\).
We compute 
\begin{equation*}
    \begin{aligned}
        \frac{x_{\ell}}{r_{\ell}^2}&=\frac{x_{j}+a_{j\ell}}{\left(x_{j}+a_{j\ell}\right)^2+\left(y_{j}+b_{j\ell}\right)^2}\\
        &=\frac{x_{j}+a_{j\ell}}{r_{j}^2+d_{j\ell}^2}\frac{1}{1+\frac{2a_{j\ell}x_{j}+2b_{j\ell}y_{j}}{r_{j}^2+d_{j\ell}^2}}\\ 
        &\sim\frac{x_{j}+a_{j\ell}}{r_{j}^2+d_{j\ell}^2}-\frac{\left(x_{j}+a_{j\ell}\right)\left(2a_{j\ell}x_{j}+2b_{j\ell}y_{j}\right)}{(r_{j}^2+d_{j\ell}^2)^2}.
    \end{aligned}
\end{equation*}
Similarly,
\begin{equation*}
    \frac{y_{\ell}}{r_{\ell}^2}\sim \frac{y_{j}+b_{j\ell}}{r_{j}^2+d_{j\ell}^2}-\frac{\left(y_{j}+b_{j\ell}\right)\left(2a_{j\ell}x_{j}+2b_{j\ell}y_{j}\right)}{(r_{j}^2+d_{j\ell}^2)^2}.
\end{equation*}
Denote by 
\begin{equation}\label{Ajl}
    A_{j\ell}=A^{(1)}_{j\ell}+A^{(2)}_{j\ell}+A^{(3)}_{j\ell}+\cdots,
\end{equation}
where $A^{(k)}_{j\ell}=O(\alpha^{-k})$. Let us
introduce the notations $A^{(k)}_{j\ell,m}$ and $B^{(k)}_{j\ell,m}$, which correspond respectively to the Fourier mode \(m\) parts of \(A^{(k)}_{j\ell}\) and \(B^{(k)}_{j\ell}\).
Then for $k=1,2$, we have 
\begin{equation*}
    \begin{aligned}
        A^{(1)}_{j\ell}&=\frac{1}{2}\frac{ a_{j\ell}}{r^2_j+d^2_{j\ell}}\left(e^{i\theta_{j}}+e^{-i\theta_{j}}\right)-\frac{i}{2}\frac{b_{j\ell}}{r^2_j+d^2_{j\ell}}\left(e^{i\theta_{j}}-e^{-i\theta_{j}}\right)\\
        &=\frac{1}{2}\frac{ a_{j\ell}-i b_{j\ell}}{r^2_j+d^2_{j\ell}}e^{i\theta_j}+\frac{1}{2}\frac{ a_{j\ell}+i b_{j\ell}}{r^2_j+d^2_{j\ell}}e^{-i\theta_j}\\
        &=A^{(1)}_{j\ell,1}e^{i\theta_j}+A^{(1)}_{j\ell,-1}e^{-i\theta_j}.
    \end{aligned}
\end{equation*}
Note that 
$$
\begin{aligned}
    &\left[\frac12\frac{r_{j}}{r_{j}^2+d_{j\ell}^2}-\frac{a_{j\ell}^2r_{j}}{(r_{j}^2+d_{j\ell}^2)^2}\right]-\left[\frac{b_{j\ell}^2r_j}{(r_{j}^2+d_{j\ell}^2)^2}-\frac{1}{2}\frac{r_{j}}{r_{j}^2+d_{j\ell}^2}\right]\\
    &=\frac{r^3_{j}}{(r_{j}^2+d_{j\ell}^2)^2}=A^{(4)}_{j\ell,0}.
\end{aligned}
$$
Hence,
\begin{equation*}
    \begin{aligned}
        A^{(2)}_{j\ell}&=\left[\frac14\frac{r_{j}}{r_{j}^2+d_{j\ell}^2}-\frac12\frac{a_{j\ell}^2r_{j}}{(r_{j}^2+d_{j\ell}^2)^2}\right]\left(e^{i\theta_{j}}+e^{-i\theta_{j}}\right)^2\\
        &+\left[\frac{1}{2}\frac{b_{j\ell}^2r_j}{(r_{j}^2+d_{j\ell}^2)^2}-\frac{1}{4}\frac{r_{j}}{r_{j}^2+d_{j\ell}^2}\right]\left(e^{i\theta_{j}}-e^{-i\theta_{j}}\right)^2\\
        &+i\frac{a_{j\ell}b_{j\ell}r_{j}}{(r_{j}^2+d_{j\ell}^2)^2}\left(e^{2i\theta_{j}}-e^{-2i\theta_{j}}\right)-A^{(4)}_{j\ell,0}\\
        &=-\frac{1}{2}\frac{r_j(a_{j\ell}-i b_{j\ell})^2}{(r_j^2+d_{j\ell}^2)^2}e^{2i\theta_j}-\frac{1}{2}\frac{r_j(a_{j\ell}+i b_{j\ell})^2}{(r_j^2+d_{j\ell}^2)^2}e^{-2i\theta_j}\\
        &=A^{(2)}_{j\ell,2}e^{2i\theta_j}+A^{(2)}_{j\ell,-2}e^{-2i\theta_j}.
    \end{aligned}
\end{equation*}
Similarly, we have 
\begin{equation}\label{Bjl}
    \begin{aligned}
        B_{j\ell}&=-\frac{i}{2}\frac{x_\ell}{r^2_\ell}(e^{i\theta_j}-e^{-i\theta_j})-\frac{1}{2}\frac{y_\ell}{r^2_\ell}(e^{i\theta_j}+e^{-i\theta_j})\\
        &=B^{(1)}_{j\ell}+B^{(2)}_{j\ell}+B^{(3)}_{j\ell}+\cdots,
    \end{aligned}
\end{equation}
where $B^{(k)}_{j\ell}=O(\alpha^{-k})$ and 
\begin{equation*}
    \begin{aligned}
        B^{(1)}_{j\ell}&=-\frac{i}{2}\frac{ a_{j\ell}}{r^2_j+d^2_{j\ell}}(e^{i\theta_j}-e^{-i\theta_j})-\frac{1}{2}\frac{b_{j\ell}}{r^2_j+d^2_{j\ell}}(e^{i\theta_j}+e^{-i\theta_j}),\\
        &=-\frac{i}{2}\frac{ a_{j\ell}-i b_{j\ell}}{r^2_j+d^2_{j\ell}}e^{i\theta_j}+\frac{i}{2}\frac{ a_{j\ell}+i b_{j\ell}}{r^2_j+d^2_{j\ell}}e^{-i\theta_j}\\
        &=B^{(1)}_{j\ell,1}e^{i\theta_j}+B^{(1)}_{j\ell,-1}e^{-i\theta_j},
    \end{aligned}
\end{equation*}
\begin{equation*}
    \begin{aligned}
        B^{(2)}_{j\ell}&=\frac{i}{2}\frac{(a^2_{j\ell}-b^2_{j\ell})r_j}{(r^2_j+d^2_{j\ell})^2}\left(e^{2i\theta_j}-e^{-2i\theta_j}\right)+\frac{a_{j\ell}b_{j\ell}r_j}{(r^2_j+d^2_{j\ell})^2}\left(e^{2i\theta_j}+e^{-2i\theta_j}\right)\\
        &=\frac{i}{2}\frac{r_j(a_{j\ell}-i b_{j\ell})^2}{(r_j^2+d_{j\ell}^2)^2}e^{2i\theta_j}-\frac{i}{2}\frac{r_j(a_{j\ell}+i b_{j\ell})^2}{(r_j^2+d_{j\ell}^2)^2}e^{-2i\theta_j}\\
        &=B^{(2)}_{j\ell,2}e^{2i\theta_j}+B^{(2)}_{j\ell,-2}e^{-2i\theta_j}.
    \end{aligned}
\end{equation*}
Note that the following relations hold:
\begin{equation}\label{AB}
    B^{(1)}_{j\ell,1}=-iA^{(1)}_{j\ell,1},\quad B^{(1)}_{j\ell,-1}=iA^{(1)}_{j\ell,-1},\quad B^{(2)}_{j\ell,2}=-iA^{(2)}_{j\ell,2},\quad B^{(2)}_{j\ell,-2}=iA^{(2)}_{j\ell,-2}.
\end{equation}
For convenience in the subsequent computations, we record the following important identity:
\begin{equation}\label{m0}
        A_{j k}^{(1)} A_{j \ell}^{(1)}+B_{j k}^{(1)} B_{j \ell}^{(1)}=\frac{a_{jk}a_{j\ell}+b_{jk}b_{j\ell}}{(r_j^2+d_{j k}^2)(r_j^2+d_{j \ell}^2)},
\end{equation}
which implies that the left-hand side contains only the Fourier mode $m=0$.

\bigskip

By the error estimate in Lemma \ref{Error} and the vortex configuration conditions, we assume that $|\psi|=O(\alpha^{-2})$.
Then for each fixed $j$ and $z\in\{|z_j|\leq C\}$, there exists a constant $\vartheta_j\in\mathbb{R}$ such that 
$$
\phi=u^*\psi=\Psi_{\sigma_j}(z_j)e^{i\vartheta_j}\psi(z_j+q_j)+O(\alpha^{-3}).
$$
Denote by $\phi_j(z_j)=e^{-i\vartheta_j}\phi(z)$ and $\psi_j(z_j)=\psi(z)$. Here 
\begin{equation}\label{vartheta}
    \vartheta_j=\mathrm{arg}\left(\lim_{z_j\to 0}\frac{u^*(z_j+q_j)}{\Psi_{\sigma_j}(z_j)}\right),
\end{equation}
and \(\vartheta_j\) may be viewed as the phase shift at the \(j\)-th vortex induced by the positions of the other vortices.

We agree that the operators $\mathbb{L}_{\sigma_j}$ and $\mathcal{L}_{\sigma_j}$ are computed in the polar coordinates $(r_j,\theta_j)$ (note that we are now considering only the local region near a single vortex), then we have
\begin{equation*}
    \begin{aligned}
        \mathbb{L}^*\psi\sim\mathbb{L}_{\sigma_j}\psi&=\left(\Phi_1(r_j)\right)^{-1}\mathcal{L}_{\sigma_j}\left(\Phi_1(r_j)\psi_j(z_j)\right)\\
&=\left(\Phi_1(r_j)\right)^{-1}\mathcal{L}_{\sigma_j}\left(e^{-i\sigma_j\theta_j}\phi_j(z_j)\right)+O(\alpha^{-3}).
    \end{aligned}
\end{equation*}
Combining the expansions of \(A_{j\ell}\) and \(B_{j\ell}\) with respect to the parameter \(\alpha\) (see \eqref{Ajl} and \eqref{Bjl}), from now on, for \(E_{st}\) defined in Lemma \ref{Error}, we write 
$$
E^{(k)}_{st}:=O(\alpha^{-k})\text{ terms of } \Phi_1(r_j)E_{st}.
$$
Similarly, we define
$$
E^{(k)}:=O(\alpha^{-k})\text{ terms of } \Phi_1(r_j)(E_1+i E_2).
$$
Therefore, to eliminate the $O(\alpha^{-2})$ terms and improve the approximate solution, at least locally, it suffices to solve the following linearized problem
\begin{equation}\label{linearized-2}
    \mathcal{L}_{\sigma_j}\left(\xi^{(2)}_j(z_j)\right)=E^{(2)}.
\end{equation}
Then we denote the approximate solution obtained after the first improvement by 
\begin{equation}\label{improve}
    u^*_1=u^*+\sum_{j=0}^n e^{i\sigma_j\theta_j}e^{i\vartheta_j}\xi^{(2)}_j(z_j).
\end{equation}
Strictly speaking, some cutoff functions may be needed. Here, however, we only give a heuristic discussion of the Lyapunov-Schmidt reduction procedure.
This improvement of the approximate solution can be iterated indefinitely.

From the perspective of Lyapunov-Schmidt reduction, it suffices to improve only the component orthogonal to the translational kernel; that is, we need only consider the Fourier modes \(m\neq \pm 1\). From the previous computations, $E^{(1)}$ is determined by the individual factors \(A_{j\ell}\) and \(B_{j\ell}\), both of which correspond to the Fourier modes \(m=\pm1\). This is why, in Eq.~\eqref{linearized-2}, we first consider \(k=2\) instead of $k=1$.

We now compute the expansion of \(E^{(2)}\) in detail. Note that 
$$
\frac{\Phi'_1(r_j)}{\Phi_1(r_j)}=\frac{4}{r_j (r^2_j+4)},\quad \frac{\Phi'_1(r_j)^2}{\Phi_1(r_j)^2}-\frac{1}{r^2_j}=-\frac{r^2_j+8}{(r^2_j+4)^2},
$$
and 
\begin{equation*}
    \begin{aligned}
        |u^*|^2&=\Phi_1(r_j)^2\prod_{\ell\ne j}\Phi_1(r_\ell)^2\\
        &=\Phi_1(r_j)^2\prod_{\ell\ne j}\left(1-\frac{4}{r^2_\ell+4}\right)\\
        &=\Phi_1(r_j)^2\left(1-4\sum_{\ell\ne j}r_\ell^{-2}\right)+O(\alpha^{-4}).
    \end{aligned}
\end{equation*}
Therefore, $E^{(2)}_{12}=0$ and 
\begin{equation*}
    \begin{aligned}
        E^{(2)}_{11}&=-\frac{4\Phi_1\left(r_{j}\right)}{2-\Phi_1\left(r_{j}\right)^2}\frac{\sigma_{j}}{r_{j}}\sum_{\ell\ne j}\sigma_{\ell}A_{j\ell}^{\left(2\right)}\\
        &\quad-\frac{2\Phi_1\left(r_{j}\right)}{2-\Phi_1\left(r_{j}\right)^2}\sum_{k,\ell\neq j;k\ne \ell}\sigma_{k}\sigma_{\ell}\left[A_{jk}^{\left(1\right)}A_{j\ell}^{\left(1\right)}+B_{jk}^{\left(1\right)}B_{j\ell}^{\left(1\right)}\right]\\
        &=-\frac{4(r^2_j+4)^{1/2}}{r^2_j+8}\sigma_j\sum_{\ell\ne j}\sigma_{\ell}A_{j\ell}^{\left(2\right)}\\
        &\quad-\frac{2r_j(r^2_j+4)^{1/2}}{r^2_j+8}\sum_{k,\ell\neq j;k\ne \ell}\sigma_{k}\sigma_{\ell}\left[A_{jk}^{\left(1\right)}A_{j\ell}^{\left(1\right)}+B_{jk}^{\left(1\right)}B_{j\ell}^{\left(1\right)}\right],
    \end{aligned}
\end{equation*}
\begin{equation*}
    \begin{aligned}
        E^{(2)}_{13}&=-\Phi_1(r_j)\left[\frac{2}{2-\Phi_1(r_j)^2\left(1-4\sum_{\ell\ne j}r^{-2}_\ell\right)}- \frac{2}{2-\Phi_{1}(r_j)^2}\right]\frac{r^2_j+8}{(r^2_j+4)^2}\\
        &\quad-\Phi_1(r_j)\left(\frac{\Phi _1\left( r_j \right) ^2}{2-\Phi_{1}(r_j)^2}-1\right)\sum_{\ell\ne j}\frac{r^2_\ell+8}{(r^2_\ell+4)^2}+O(\alpha^{-4})\\
        &=\Phi_1(r_j)\left[\frac{8\Phi_1(r_j)^2}{\left(2-\Phi_1(r_j)^2\right)^2}\frac{r^2_j+8}{(r^2_j+4)^2}+\frac{2-2\Phi_1(r_j)^2}{2-\Phi_1(r_j)^2}\right]\sum_{\ell\ne j}r_\ell^{-2}+O(\alpha^{-4})\\
        &=\frac{16 r_j(r_j^2+2)}{(r_j^2+4)^{3 / 2}(r_j^2+8)}\sum_{\ell\ne j}\left[\left(A^{(1)}_{j\ell}\right)^2+\left(B^{(1)}_{j\ell}\right)^2\right],
    \end{aligned}
\end{equation*}
\begin{equation*}
\begin{aligned}
E^{(2)}_{14}
={}&\frac{1}{2}\Phi_1(r_j)
\Bigl[1-\Phi_1(r_j)^2+4\Phi_1(r_j)^2\sum_{\ell\ne j}r_\ell^{-2}\Bigr]
\Bigl[2-\Phi_1(r_j)^2+4\Phi_1(r_j)^2\sum_{\ell\ne j}r_\ell^{-2}\Bigr] \\
&-\frac{1}{2}\Phi_1(r_j)
\bigl(1-\Phi_1(r_j)^2\bigr)
\bigl(2-\Phi_1(r_j)^2\bigr) \\
&-\frac{1}{2}\Phi_1(r_j)
\sum_{\ell\ne j}
\bigl(1-\Phi_1(r_\ell)^2\bigr)
\bigl(2-\Phi_1(r_\ell)^2\bigr)
+O(\alpha^{-4}) \\
={}&
\bigl[-4\Phi_1(r_j)^5+6\Phi_1(r_j)^3-2\Phi_1(r_j)\bigr]
\sum_{\ell\ne j}r_\ell^{-2}
+O(\alpha^{-4}) \\
={}&
\frac{8r_j(r_j^2-4)}{(r_j^2+4)^{5/2}}
\sum_{\ell\ne j}
\left[
\bigl(A^{(1)}_{j\ell}\bigr)^2+
\bigl(B^{(1)}_{j\ell}\bigr)^2
\right].
\end{aligned}
\end{equation*}
and
\begin{equation*}
        \begin{aligned}
            E^{(2)}_2&=E^{(2)}_{21}+E^{(2)}_{22}\\
            &=\frac{4\Phi'_1(r_j)}{2-\Phi_1(r_j)^2}\sum_{\ell\ne j}\sigma_\ell B^{(2)}_{j\ell}\\
            &=\frac{16}{(r^2_j+4)^{1/2}(r^2_j+8)}\sum_{\ell\ne j}\sigma_\ell B^{(2)}_{j\ell}.
        \end{aligned}
\end{equation*}
For each pair of indices $s,t$, we consider the Fourier decomposition
\[
E^{(2)}_{st}
=
\sum_{m\in\mathbb Z}
E^{(2)}_{st,m}e^{im\theta_j}.
\]
For fixed $j$, we define 
$$
\mathcal{E}^{(2)}_{j,m}:=\sum_{s,t}E^{(2)}_{st,m}.
$$
Combining the Fourier expansions of
$A^{(k)}_{j\ell}$, $B^{(k)}_{j\ell}$, and $E^{(2)}_{st}$,
we obtain
\[
E^{(2)}_1
=
\sum_{t\in\{1,3,4\}}
\sum_{m\in\{0,\pm2\}}
E^{(2)}_{1t,m}e^{im\theta_j},
\qquad
E^{(2)}_2
=
\sum_{m=\pm2}
E^{(2)}_{2,m}e^{im\theta_j}.
\]

\bigskip

We now return to the \((1,3)\)-vortex configuration considered before, namely, \(q_j=p_j\) for $j=0,1,2,3$.
In this case, by definition \eqref{vartheta} and a direct computation, we obtain
$$
\vartheta_0=\frac{3\pi}{4},\quad \vartheta_1=-\frac{3\pi}{4},\quad \vartheta_2=-\frac{\pi}{12},\quad \vartheta_3=\frac{7\pi}{12}.
$$
Now for $j=0$, let us consider the asymptotic expansion of the difference between the exact solution $\Psi_{2,\alpha}$ (given by Theorem \ref{thm}) and approximate solution $u^*$ as \(\alpha\to\infty\):
$$
\Psi_{2,\alpha}(z)-u^*(z)=c_{02}(x,y)\alpha^{-2}+c_{03}(x,y)\alpha^{-3}+\cdots,
$$
where 
\begin{equation*}
    c_{02}(x, y)=\frac{3[-x+y+i(x+y)]}{\sqrt{2\left(4+x^2+y^2\right)}}=3 e^{3 \pi i / 4} \frac{\bar{z}}{\sqrt{r^2+4}},\quad z=x+iy=r e^{i\theta}.
\end{equation*}
By \eqref{improve}, we denote by
\begin{equation}\label{xi2}
    \xi^{(2)}_0(z):=c_{02}(x,y)e^{-i\sigma_0 \theta}e^{-i\vartheta_0}\alpha^{-2}=\frac{3r}{\sqrt{r^2+4}}=3\Phi_1(r)\alpha^{-2}.
\end{equation}
Observe that it contains only the Fourier mode $m=0$, and hence
\begin{equation}\label{xi}
    \xi^{(2)}_0(z)=\sum_{m\in\mathbb{Z}}\xi^{(2)}_{0,m}(r)e^{im\theta}=\xi^{(2)}_{0,0}(r).
\end{equation}
Next, we will verify that 
$$
\mathcal{L}_{-1}(\xi^{(2)}_0)\sim E^{(2)}_{1,0}=E^{(2)}_{11,0}+E^{(2)}_{13,0}+E^{(2)}_{14,0}.
$$
Note that 
$$
\frac{16 r(r^2+2)}{(r^2+4)^{3 / 2}(r^2+8)}+\frac{8r(r^2-4)}{(r^2+4)^{5/2}}=\frac{8 r\left(3 r^4+16 r^2-16\right)}{\left(r^2+4\right)^{5 / 2}\left(r^2+8\right)}.
$$
Hence, 
\begin{equation*}
    \begin{aligned}
        E^{(2)}_{1,0}&=-\frac{2r(r^2+4)^{1/2}}{r^2+8}\sum_{k,\ell\neq 0;k\ne \ell}\sigma_{k}\sigma_{\ell}\left(\frac{a_{0k}}{r^2+d^2_{0k}}\frac{ a_{0\ell}}{r^2+d^2_{0\ell}}+\frac{b_{0k}}{r^2+d^2_{0k}}\frac{b_{0\ell}}{r^2+d^2_{0\ell}}\right)\\
        &\quad+\frac{8 r\left(3 r^4+16 r^2-16\right)}{\left(r^2+4\right)^{5 / 2}\left(r^2+8\right)}\sum_{\ell\neq 0}\frac{d^2_{0\ell}}{(r^2+d^2_{0\ell})^2}.
    \end{aligned}
\end{equation*}
We define the $O(\alpha^{-2})$ approximation of $E^{(2)}_{1,0}$ near the vortex $p_0$ by 
\begin{equation}\label{F}
    \begin{aligned}
        F^{(2)}_{1,0}&=-\frac{2r(r^2+4)^{1/2}}{r^2+8}\sum_{k,\ell\neq 0;k\ne \ell}\sigma_{k}\sigma_{\ell}\left(\frac{a_{0k}}{d^2_{0k}}\frac{ a_{0\ell}}{d^2_{0\ell}}+\frac{b_{0k}}{d^2_{0k}}\frac{b_{0\ell}}{d^2_{0\ell}}\right)\\
        &\quad+\frac{8 r\left(3 r^4+16 r^2-16\right)}{\left(r^2+4\right)^{5 / 2}\left(r^2+8\right)}\sum_{\ell\neq 0}\frac{1}{d^{2}_{0\ell}}.
    \end{aligned}
\end{equation}
For different indices $k,\ell\neq 0$, it follows from \eqref{pj} and the definition that
$$
d_{0 k}^2=2 \alpha^2, \quad a_{0 k} a_{0 \ell}+b_{0 k} b_{0 \ell}=p_k \cdot p_{\ell}=-\alpha^2.
$$
Hence, 
$$
\sum_{\substack{k, \ell \neq 0;k \neq \ell}} \sigma_k \sigma_{\ell} \frac{a_{0 k} a_{0 \ell}+b_{0 k} b_{0 \ell}}{d_{0 k}^2 d_{0 \ell}^2}=-\frac{3}{2 \alpha^2},
$$
and
$$
\sum_{\ell \neq 0} \frac{1}{d^2_{0 \ell}}=\frac{3}{2 \alpha^2}.
$$
Substituting them into \eqref{F}, we obtain 
$$
F^{(2)}_{1,0}=\frac{3 r^3\left(r^4+24 r^2+112\right)}{\left(r^2+4\right)^{5 / 2}\left(r^2+8\right)}\alpha^{-2}.
$$
By the definition of the operator $\mathcal{L}_n$ (see \eqref{Ln}), a direct computation using numerical software gives
\begin{equation}\label{3phi}
    \mathcal{L}_{-1}(3\Phi_1(r))=\alpha^{2}F^{(2)}_{1,0}.
\end{equation}
Therefore, the choice of the improvement term $\xi^{(2)}_{0,0}$ (see \eqref{xi2}-\eqref{xi}) is justified within the Lyapunov--Schmidt reduction framework by \eqref{3phi}. To see that the Fourier mode $m=\pm 2$ components of the first improvement term $\xi^{(2)}_{0,m}$ near $p_0$ vanish, it suffices to check whether the Fourier mode $m=\pm 2$ components of the $O(\alpha^{-2})$ terms of the error also vanish near $p_0$. 
By \eqref{AB}, \eqref{m0} and definition, we compute 
\begin{equation*}
    \begin{aligned}
        \mathcal{E}^{(2)}_{j,2}&=E^{(2)}_{11,2}+iE^{(2)}_{2,2}\\
        &=\left[-\frac{4 (r_j^2+4)^{1/2}}{r_j^2+8} \sigma_j+\frac{16}{(r_j^2+4)^{1/2}(r^2_j+8)}\right] \sum_{\ell \neq j} \sigma_{\ell} A_{j \ell, 2}^{(2)}\\
        &=-\frac{2 r_j\left[4-\sigma_j(r_j^2+4)\right]}{(r_j^2+4)^{1/2}(r_j^2+8)} \sum_{\ell \neq j} \sigma_{\ell} \frac{\left(a_{j \ell}-i b_{j \ell}\right)^2}{(r_j^2+d_{j \ell}^2)^2}.
    \end{aligned}
\end{equation*}
Similarly, we have 
\begin{equation*}
    \begin{aligned}
        \mathcal{E}^{(2)}_{j,-2}&=E^{(2)}_{11,-2}+iE^{(2)}_{2,-2}\\
        &=\frac{2 r_j\left[4+\sigma_j(r_j^2+4)\right]}{(r_j^2+4)^{1/2}(r_j^2+8)} \sum_{\ell \neq j} \sigma_{\ell} \frac{\left(a_{j \ell}+i b_{j \ell}\right)^2}{(r_j^2+d_{j \ell}^2)^2}.
    \end{aligned}
\end{equation*}
Substituting the vortex data of the $(1,3)$-configuration, for $j=0$, a direct computation shows that
$$
\sum_{\ell=1}^3 \frac{\left(a_{0 \ell}\pm i b_{0 \ell}\right)^2}{(r^2+d_{0 \ell}^2)^2}=0,
$$
which implies that $\mathcal{E}^{(2)}_{0,\pm 2}=0$ and hence $\xi^{(2)}_{0,\pm 2}=0$.

After a change of coordinates, the above discussion of the improvement of the approximate solution within the Lyapunov-Schmidt reduction framework can also be carried out at the remaining vortex points. For comparison, it suffices to consider the formal expansion of 
$$
\Psi_{2,\alpha}(z_j+p_j)-u^*(z_j+p_j)=c_{j2}(x_j,y_j)\alpha^{-2}+c_{j3}(x_j,y_j)\alpha^{-3}+\cdots, \quad \text{as } \alpha \to\infty,
$$
where $z_j=z-p_j=r_j e^{i\theta_j}=x_j+i y_j$. The subsequent local approximate liftings near the vortices are (formally) similar to the case $j=0$, and we omit the details here. 

This analysis already indicates that a Lyapunov-Schmidt reduction for this equation is quite delicate, especially regarding the decay behavior of the perturbation term away from the vortex points (for instance, $c_{j3}$ itself grows in $x,y$). An interesting feature is that, at orders $\alpha^{-2}$ and $\alpha^{-3}$, the Ginzburg-Landau equation and the complex sine-Gordon II equation are quite similar. However, in the higher-order perturbation terms, one should be able to see their differences, which reflect precisely why the complex sine-Gordon II equation is integrable while the Ginzburg-Landau equation is not.

\bigskip

	\begin{center}
		{\bf Acknowledgments}
	\end{center}
	
The research of S. Chen was supported by National Key R\&D Program program of China 2022YFA1005400, National Science Fund for Distinguished Young Scholars (No.12225111), NSFC
No.12426202.	The research of C.F. Gui is supported by University of Macau research grants CPG2024-00016-FST, CPG202500032-FST, SRG2023-00011-FST, MYRG-GRG2023-00139-FST-UMDF, UMDF Professorial Fellowship of Mathematics, Macao SAR FDCT0003/2023/RIA1 and Macao SAR FDCT0024/2023/RIB1, NSFC No. 12531010. The research of Y. Liu is supported by  NSFC No. 12471204.  The research of W. Yang is partially supported by National Key R\&D Program of China 2022YFA1006800, NSFC No.12171456, 12271369 and 12531010, FDCT No. 0070/2024/RIA1, Start-up Research Grant No. SRG2023-00067-FST, Multi-Year Research Grant No. MYRG-GRG2024-00082-FST-UMDF, MYRG-GRG2025-00051-FST and UMDF No. TISF/2025/006/FST.

\bibliographystyle{amsplain}

\end{document}